\documentclass[11pt]{amsart}

\usepackage[T1]{fontenc}
\usepackage{hyperref}
\usepackage{cite}
\usepackage{amssymb}
\usepackage{amsmath}
\usepackage{amsthm}
\usepackage{amsfonts}
\usepackage{graphicx}
\usepackage{bbm}
\usepackage{xcolor}
\usepackage[toc,page]{appendix}
\usepackage{mathtools}
\usepackage{comment}

\usepackage{yhmath}
\usepackage{subfigure}
\usepackage{mathtools} 
\usepackage{enumerate}

\usepackage[all]{xy}
\usepackage{color}

\usepackage{marginnote}
\usepackage{verbatim}

\usepackage{todonotes}

\newtheorem*{theorem*}{Theorem}

\newtheorem{theorem}{Theorem}[section]

\newtheorem{dfn}[theorem]{Definition}
\newtheorem{proposition}[theorem]{Proposition}

\newtheorem{lemma}[theorem]{Lemma}
\newtheorem{claim}[theorem]{Claim}
\newtheorem{corollary}[theorem]{Corollary}

\newtheorem{conjecture}[theorem]{Conjecture}

\theoremstyle{remark}

\newtheorem{rem}[theorem]{Remark}

\theoremstyle{definition}
\newtheorem{remark}[theorem]{Remark}

\newtheorem{definition}[theorem]{Definition}

\newtheorem{question}[theorem]{Question}

\newtheorem{example}[theorem]{Example}
\numberwithin{equation}{section}

\renewcommand{\phi}{\varphi}

\newcommand{\cA}{\mathcal{A}}

\newcommand{\masha}[1]{{\color{blue} Masha says: #1}}

\def\XXint#1#2#3{{\setbox0=\hbox{$#1{#2#3}{\int}$}
\vcenter{\hbox{$#2#3$}}\kern-.5\wd0}}

\newcommand{\ad}{\operatorname{ad}}

\newcommand{\N}{\mathbb N}

\newcommand{\R}{\mathbb R}

\newcommand{\g}{\mathfrak{g}}

\newcommand{\HH}{\mathbb H}

\renewcommand{\ker}{\operatorname{Ker}}

\newcommand{\Lie}{\operatorname{Lie}}

\newcommand{\uno}{{\mathbbm{1}}}

\newcommand{\norm}[1]{\left\Vert#1\right\Vert}

\def\eps{\epsilon}

\def \dd {\mathrm{d}}

\newcommand{\vol}{\operatorname{vol}}

\newcommand{\sus}{\subseteq}
\newcommand{\at}[1]{\raise-.5ex\hbox{\ensuremath|}_{#1}}

\begin{document}

\title[Null sets in infinite-dimensional Carnot groups]{Notions of null sets in infinite-dimensional Carnot groups}

\author[Eldredge]{Nathaniel Eldredge{$^{\ast}$}}
\thanks{\footnotemark {$\ast$} Research was supported in part by
    Simons Foundation Grant \#355659.}
\address{N. Eldredge;
Department of Mathematical Sciences
\\
University of Northern Colorado
\\
Greeley, CO 80369, U.S.A.
}
\email{neldredge@unco.edu}

\author[Gordina]{Maria Gordina{$^{\dag}$}}
\thanks{\footnotemark {$\dag$} Research was supported in part by NSF Grant DMS-1954264.}
\address{M. Gordina;
Department of Mathematics
\\
University of Connecticut
\\
Storrs, CT 06269,  U.S.A.
}
\email{maria.gordina@uconn.edu}

\author[Le Donne]{Enrico Le Donne{$^{\ddag }$}}
\thanks{\footnotemark {$\ddag$} Research
was   supported  in part by
the Swiss National Science Foundation (grant 200021-204501 `\emph{Regularity of sub-Riemannian geodesics and applications}'),
by the European Research Council  (ERC Starting Grant 713998 GeoMeG `\emph{Geometry of Metric Groups}'),
and  by the Academy of Finland  (grant 322898 `\emph{Sub-Riemannian Geometry via Metric-geometry and Lie-group Theory}').
 }

\address{E. Le Donne;
University of Fribourg, Chemin du Mus\'ee~23, 1700 Fribourg, Switzerland \& Department of Mathematics and Statistics \\
P.O. Box 35, FI-40014 \\
University of Jyv\"askyl\"a, Finland}
\email{enrico.ledonne@unifr.ch}

\author[Li]{Sean Li}
\address{S. Li;
Department of Mathematics
\\
University of Connecticut
\\
Storrs, CT 06269,  U.S.A.
}
\email{sean.li@uconn.edu}

\begin{abstract}
We study several notions of null sets on infinite-dimensional Carnot groups. We prove that
a set is  Aronszajn null if and only if it is null with respect to measures that are convolutions of absolutely continuous (CAC) measures on Carnot subgroups. The CAC measures are the non-abelian analogue of cube measures. In the case of infinite-dimensional Heisenberg-like groups we also show that being null in the previous senses is equivalent to being null for all heat kernel measures.  Additionally, we show that infinite-dimensional Carnot groups that have locally compact commutator subgroups have the structure of Banach manifolds. There are a number of open questions included as well. 
\end{abstract}



\date{\today}
\maketitle
\setcounter{tocdepth}{3}

\tableofcontents


\section{Introduction}

One of the main challenges in analysis on infinite-dimensional Banach spaces is that there is no canonical reference measure, and therefore, no corresponding canonical notion of a null set.  Instead, several different notions have been proposed.  These have important applications, such as in generalizations of Rademacher's theorem which show that the points of non-differentiability of a Lipschitz function are in some sense a null set \cite{LindenstraussPreiss2001}.

This raises a natural question: how do different notions of null sets relate to one another?  In this direction, we focus on a theorem due to M.~Cs\"{o}rnyei \cite{Csornyei1999a}, which states that three of them coincide: specifically, the notions of Aronszajn null, cube null, and Gaussian null sets.

A Borel set $E$ of a separable real Banach space $X$ is said to be \emph{Aronszajn null} if for every sequence $\{v_i\}_{i \in \N} \subset X$ with a dense span, the set $E$ can be decomposed into Borel subsets $\bigcup_{i=1}^\infty E_i$ so that
\begin{align*}
  m(\{t \in \R : x + tv_i \in E_i\}) = 0, \text{ for all } x \in X, i \in \N,
\end{align*}
where $m$ is the Lebesgue measure on $\R$.  Next, non-degenerate \emph{cube measures} on $X$ are distributions of $X$-valued random variables of the form $x + \sum_{i=1}^\infty X_i v_i$, where $x \in X$, $\{X_i\}_{i = 1}^\infty$ are iid uniform random variables on $[-1,1]$, and $\{v_i\}_{i=1}^\infty \subset X$ are sequences of vectors whose span is dense in $X$ and $\sum_{i=1}^\infty \|v_i\| < \infty$.  A Borel set $E$ is then said to be \emph{cube null} if it is null with respect to every non-degenerate cube measure.  Finally, a Borel set is \emph{Gaussian null} if it is null with respect to every non-degenerate Gaussian measure.  It is not hard to see that for each of these notions, the corresponding collection of null sets is a $\sigma$-ideal containing no nonempty open sets; all notions of null sets studied here will have this property.

In the present work we consider different notions of null sets when  Banach spaces are  replaced by a special class of infinite-dimensional topological groups. In such a setting we need to introduce new notions of null sets in the absence of some linear structure as the one needed to define cube and Gaussian measures. We first recall a class of infinite-dimensional groups for which the notion of Aronszajn null sets has been studied. In \cite{LeDonneLiMoisala2021}, the authors introduced \emph{infinite-dimensional Carnot groups}, a class of infinite-dimensional topological groups modeled after Carnot groups and Banach spaces.  We recall their definition, starting with their notion of (metric) scalable groups.

\begin{definition}[Scalable group] \label{def:scalable:group}
A  \emph{scalable group} is a pair $\left( G, \delta \right)$, where $G$ is a topological group and $\delta \colon \mathbb{R} \times G \to  G$  is a continuous  map such that  $\delta_{\lambda}:=\delta(\lambda, \cdot) \in	\operatorname{Aut}(G)$ for all $\lambda \in \R \setminus \{ 0 \} $,
\begin{equation}\label{eq:composition_of_dilations}
 \delta_\lambda \circ \delta_\mu =\delta_{\lambda \mu} \text{ for all } \lambda, \mu\in \R,
\end{equation}
and $\delta_0 \equiv e_G$, where $e_G$ is the identity element of $G$.
\end{definition}
It follows that $\delta_1$ is the identity map of $G$. We say that $H$ is a \emph{scalable subgroup} of a scalable group $\left(G, \delta \right)$ if $H$ is a closed subgroup of $G$  and $\delta_{\lambda}\left( H \right)= H$ for all $\lambda \in  \mathbb{R}\backslash \left\{ 0 \right\}$.

\begin{definition}[Metric scalable group] \label{def:metric:scalable:group}
A \emph{metric scalable group} is a triple $(G, \delta, d)$, where  $(G,\delta)$ is a scalable group and $d$ is an admissible left-invariant distance on $G$ such that
  \[
    d(\delta_t(p),\delta_t(q)) = |t| d(p,q), \qquad \text{ for all } t\in \R.
  \]
\end{definition}
By \emph{admissible} we mean that the metric induces the original topology on $G$.

We now introduce the notion of a filtration by (finite-dimensional) Carnot subgroups.

\begin{definition}[Filtration by Carnot subgroups]\label{d.filtration}
We say that a  scalable group $G=\left( G, \delta \right)$ is \emph{filtrated by Carnot subgroups} if there exists a sequence $\left\{ N_m \right\}_{m \in \N}$ of scalable subgroups of $G$ such that each $N_m$ has a Carnot group structure, $N_m<N_{m+1}$, and $G$ is the closure of $\cup_{m\in\N} N_m$.
In this case, we say that the sequence $\left\{ N_m \right\}_{m \in \N}$ is \emph{a filtration of the  scalable group $G$ by Carnot subgroups}.
\end{definition}

Finally, we give the definition of infinite-dimensional Carnot groups.

\begin{definition}\label{def:infinite-dim_Carnot_gp}
We call a complete metric scalable group that admits a filtration by Carnot subgroups an \emph{infinite-dimensional Carnot group}.
\end{definition}

One of the goal of \cite{LeDonneLiMoisala2021} was to extend the notion of Aronszajn null sets to these infinite-dimensional Carnot groups as we define in Definition~\ref{d.ArNullCarnot}.  We will recall properties of these groups in Section~\ref{s.InfDimCarnotGroups}, but we remark now that these groups can be wild as they need not be nilpotent nor Banach manifolds.  We will, however, give a sufficient condition for when such groups become Banach manifolds in Theorem~\ref{th:IDCG-Banach}.

In a parallel line of research, there have been studies of Brownian motion on \emph{infinite-dimensional Heisenberg-like groups}, which were introduced in \cite{DriverGordina2008} and turn out to be a subclass of the aforementioned infinite-dimensional Carnot groups.  The endpoint distributions of these Brownian motions are called \emph{heat kernel measures}.  As heat kernel measures and Gaussian measures are the same objects in the Banach space setting, heat kernel measures can be viewed as a generalization of Gaussian measures in the non-commutative setting.

In this paper we focus on \emph{hypoelliptic} heat kernel measures, which were studied in \cite{BaudoinGordinaMelcher2013, DriverEldredgeMelcher2016}.  Using regularity results from those articles, we show that the \emph{heat kernel null} sets, i.e. those which are null for every hypoelliptic heat kernel measure, coincide with the Aronszajn null sets of \cite{LeDonneLiMoisala2021}.

\begin{theorem*}[Theorem~\ref{A-iff-HK}]
A Borel set in an infinite-dimensional Heisenberg-like group is Aronszajn null if and only if it is heat kernel null.
\end{theorem*}

However, in the setting of general infinite-dimensional Carnot groups, the additional structure needed to construct Brownian motion in a canonical way is not present.  Thus, we introduce another notion of null sets of infinite-dimensional Carnot groups, namely, \emph{CAC null}.  These are Borel sets that are null with respect to a class of measures we call \emph{CAC measures} which are given by convolutions of measures that are absolutely continuous with respect to the Haar measures on the finite-dimensional Carnot subgroups of a filtration (see Definition~\ref{d.CAC}). These measures can be thought of as rough analogues of cube measures as we explain in  Remark~\ref{r.CACversusCube}.  We will prove the following theorem.

\begin{theorem*}[Theorem~\ref{th:aron-cac}]
A Borel set in an infinite-dimensional Carnot group is Aronszajn null if and only if it is CAC null.
\end{theorem*}

While the focus of our paper is on different notions of null sets, we also establish connections between infinite-dimensional Carnot groups and other classes of infinite-dimensional groups such as Banach-Lie groups. Many previous infinite-dimensional versions of Carnot groups were homeomorphic to Banach spaces.  In particular, the groups studied in \cite{DriverGordina2008, Melcher2009a, MagnaniRajala2014} are Banach-Lie groups.  While most of these groups are also infinite-dimensional Carnot groups as defined in  \cite{LeDonneLiMoisala2021}, it is not true that all infinite-dimensional Carnot groups are Banach-Lie groups.  We show that in the case when the commutator subgroup of an infinite-dimensional Carnot group is locally compact, then it is a stratified Banach-Lie group (definition given in Section \ref{sec:Banach}). Note that  in \cite{Melcher2009a} heat kernel measures were constructed on higher step Banach-Lie groups, but we do not know if these measures are absolutely continuous to measures we can analyze more easily such as Gaussian-Haar measures in Section~\ref{s.HM}.

The paper is organized as follows. Section~\ref{s.InfDimCarnotGroups} reviews basic facts about in\-fi\-nite-di\-men\-sion\-al Carnot groups, and stratified Banach-Lie groups including infinite-dimensional Heisenberg-like groups. In Section~\ref{s.Quotients} we study the quotients of infinite-dimensional Carnot groups; in particular, we show that their abelianizations are Banach spaces, a fact we will use throughout the paper. In Section~\ref{s.LCCommutators} we show that an infinite-dimensional Carnot group with a locally compact commutator subgroup is a stratified Banach-Lie group. A non-commutative analogue of cube measures and their null sets are defined and studied in Section~\ref{s.CAC}. Section~\ref{s.HM} gives the necessary background for studying heat kernel measure null sets in  Section~\ref{heat-kernel}.  We end with a discussion of open questions and conjectures in Section~\ref{s.Questions}.

\section{Preliminaries}

\subsection{Infinite-dimensional Carnot groups}\label{s.InfDimCarnotGroups}

We introduced the definition of scalable and infinite-dimensional Carnot groups in the introduction.  We now introduce the notion of a (finite-dimensional) Carnot group that we skipped in the introduction.

Every (finite-dimensional) Carnot group naturally has structure of a scalable group, where by a Carnot group $G$ we mean a simply connected Lie group whose Lie algebra $\Lie(G)$ is equipped with a stratification $\Lie(G) = V_1 \oplus\cdots\oplus V_s$. The stratification of $\Lie(G)$ is unique up to an isomorphism, see \cite[Proposition 2.2.10]{BonfiglioliLanconelliUguzzoniBook} or \cite{LeDonne2017}, and it defines a family of dilations on $G$. Indeed, by \cite[Remark 1.3.26]{BonfiglioliLanconelliUguzzoniBook} there is a family of Lie group homomorphisms corresponding to the Lie algebra scalings defined by $\delta^{\ast}_\lambda(X) = \lambda^k X$ for $X \in V_{k}$ and $\lambda\in \R \setminus \{0\}$. Such a group can be metrized as a metric scalable group, and the metric is unique up to a bi-Lipschitz equivalence. Vice versa, we say that a scalable group $(G,\delta)$ has a Carnot group structure if there exists a Carnot group that is isomorphic to $G$ as a topological group and whose dilations given by the stratification coincide with $\delta$.  For the rest of this paper, any occurrence of Carnot groups will be assumed to be finite-dimensional unless explicitly stated otherwise.

Recall that infinite-dimensional Carnot groups contain a dense filtration $\{N_m\}_{m \in \N}$ by Carnot subgroups.  Necessarily, a metric scalable group that admits a filtration by Carnot subgroups is separable. Note that an infinite-dimensional Carnot group $G$ cannot be equal to its filtration  $\cup_{m\in\N}N_m$ unless $G = N_m $ for some $m \in \N$. Indeed, each $N_m \sus G$ is nowhere dense and hence the union $\cup_{m\in\N}N_m$ is of first category in $G$.

We would like to have sufficient geometric conditions for a scalable group to admit  a filtration by Carnot subgroups.  For a scalable group $G$ define its $k$\emph{th layer} as
\begin{equation}\label{V1def}
  V_{k}(G) := \{ p\in G : t \mapsto \delta_{t^{1/k}} (p) \text{ is a one-parameter subgroup}\},
\end{equation}
where by a one-parameter subgroup we mean a continuous group homomorphism from the additive group $\mathbb{R}$ to $G$. This means that for all $t,s \in \R$ and $p \in V_{k}(G)$ we have
\[
\delta_{(t+s)^{1/k}}(p) = \delta_{t^{1/k}}(p)\delta_{s^{1/k}}(p).
\]
Note that if $p\in V_1(G)$, then $ \delta_r(p) \in V_1(G) $ for all $ r\in \R $, since
\[
\delta_{t+s}(\delta_r(p)) = \delta_{t r+s  r}(p) = \delta_{t r}(p)\delta_{s r}(p) = \delta_t(\delta_r(p))\delta_s(\delta_r(p)).
\]

\begin{lemma} \label{l:V1-complete}
Let $G$ be an infinite-dimensional Carnot group, then $V_{k}(G)$ is closed for all $k$.
\end{lemma}

\begin{proof}

Let $\{v_i\}_{ i \in \mathbb{N}}$ be a   sequence in $V_k(G)$, and let  $v \in G$ for which $v_i \rightarrow  v$.
Let $t,s \in \R$, then by continuity of $\delta$ we have
  \begin{align*}
    \delta_{(t+s)^{1/k}}(v) = \lim_{i \to \infty} \delta_{(t+s)^{1/k}}(v_i) = \lim_{i \to \infty} \delta_{t^{1/k}}(v_i)\delta_{s^{1/k}}(v_i) = \delta_{t^{1/k}}(v)\delta_{s^{1/k}}(v),
  \end{align*}
  where we also used the continuity of group multiplication in the last equality.  Thus, $t \mapsto \delta_{t^{1/k}}(v)$ is a one-parameter subgroup, and so $v \in V_k(G)$.
\end{proof}

We say that a set \emph{$ A \sus G $ generates $G$ as a scalable group} or simply that \emph{$A$ generates $G$} if $G$ is the closure of the group generated by $ \{\delta_t(a) : a\in A, \, t\in \R \} $.  Note that $V_1(G)$ is completely analogous to the (exponential image of the) generating first layer of a finite-dimensional Carnot group.  Moreover, the following proposition holds.

\begin{proposition}[{\cite[Proposition 1.10]{LeDonneLiMoisala2021}}]\label{exist:filtration}
  Let $G$ be a scalable group. If $G$ admits a filtration by Carnot subgroups then $ V_1(G) $ generates $G$ as a scalable group. Vice versa, if $G$ is nilpotent, $V_1(G)$ is separable, and  $ V_1(G) $ generates $G$ as a scalable group, then $G$ admits a filtration by Carnot subgroups.
\end{proposition}
We point out that the nilpotency assumption in the previous proposition cannot be removed, since there exist scalable groups with generating first layer that do not admit filtrations (see \cite[Proposition 5.11]{LeDonneLiMoisala2021}). However, not every metric scalable group having filtrations is nilpotent, as shown in \cite[Proposition 5.10]{LeDonneLiMoisala2021}. We will discuss this relation in more detail in Section~\ref{s.LCCommutators}.

The second part of Proposition~\ref{exist:filtration} comes from the following.

\begin{proposition}[{\cite[Proposition 2.2]{LeDonneLiMoisala2021}}]\label{p:generate-Carnot}
  Let $(G,\delta)$ be a metric scalable group generated by $v_1,\ldots,v_n \in V_1(G)$.  If $G$ is nilpotent then $G$ is a Carnot group.
\end{proposition}

We would now like to define a notion of a generating sequence of $G$ that respects the filtration.  To that end, we say that a sequence $X_1,X_2,\ldots$ in $V_1(G)$ is a
  \emph{Carnot-spanning set}  if
for every $i\in \N$  the scalable group $H_i$ generated by $X_1,\ldots, X_i $ is finite-dimensional and
\begin{align*}
  G = \overline{\bigcup_{i=1}^\infty H_i}.
\end{align*}
In other words, we require that the subgroups $H_1, H_2, \ldots$ form a filtration of $G$, according to Definition \ref{d.filtration}. We shall refer to this filtration  as the
\emph{filtration associated to the Carnot-spanning set} $\mathcal{X} = (X_1,X_2,\ldots)$.

For infinite-dimensional Carnot groups, we have the two following definitions of null sets as introduced in \cite{LeDonneLiMoisala2021}.

\begin{definition}[Filtration null]
Given a  filtration $(N_m)_{m \in \N}$ by Carnot subgroups of a scalable group $G$, we say that a Borel set $\Omega\subseteq G$ is $(N_m)_m$ \emph{null} if $\Omega$ is the countable union of Borel sets $\Omega_m$ such that
\[
\vol_{N_m} ( N_m \cap (g  \Omega_m)  )  = 0,\qquad \text{ for all } m\in \N, \text{ for all } g\in G,
\]
where $\vol_{N_m}$ denotes any Haar measure on $N_m$. Finally, we say that  a Borel set $E \subset G$ is \emph{filtration null} if it is $(N_m)_m$ null for every filtration $(N_m)_m$ of $G$.
\end{definition}

\begin{definition}[Aronszajn null]\label{d.ArNullCarnot}
Given an element $Y \in V_1( G)$, we define the class $\cA(Y)$ to be the collection of  Borel sets $E \subset G$ for which
  \begin{align*}
    \mathcal L^1 (\{t \in \R : g\delta_t(Y) \in E\}) = 0, \qquad \text{ for all } g \in G;
  \end{align*}
where we denote by $\mathcal L^1$ the Lebesgue measure on $\R$.
Given a Carnot-spanning set $\mathcal{X} = (X_1,X_2, \ldots)$, we define $\cA(\mathcal{X})$ to be the collection of Borel sets $E \subset G$ that admit Borel decompositions $E = \bigcup_i E_i$ for which
  \begin{align*}
    E_i \in \cA(X_i), \qquad \text{ for every } i \geqslant 1.
  \end{align*}
  Finally, we say that a Borel set $E \subset G$ is {\em Aronszajn null} if it lies in $\cA(\mathcal{X})$ for every Carnot-spanning set $\mathcal{X}$.
\end{definition}

As Carnot-spanning sets of $V_1(G)$ and filtrations of $G$ give rise to one another, the following proposition easily shows that these two notions of negligibility are equivalent.

\begin{proposition}[see {\cite[Proposition 3.6]{LeDonneLiMoisala2021}}] \label{th:ldlm2}
Let $G$ be a metric scalable group.
Let $\mathcal{X} = (X_1,X_2,\ldots)$ be Carnot-spanning set with $H_i := \langle X_1,\ldots,X_i \rangle$.
 Then a set $\Omega\subseteq G$ is in the class $\cA(\mathcal{X})$ if and only if it is $(H_i)_{i\in\N}$ null.
\end{proposition}

This follows easily from the following proposition that will be used later.

\begin{proposition}[see {\cite[Proposition 3.12]{LeDonneLiMoisala2021}}] \label{th:ldlm}
  Let $Y_1,\ldots,Y_n$ be generating horizontal elements of a finite-dimensional Carnot subgroup $H$ of $G$.  Given a subset $E \subset G$,
  \begin{align*}
    E \in \mathcal{A}(Y_1,\ldots,Y_n) \quad \Longleftrightarrow \quad \vol_H(x^{-1}E \cap H) = 0, \quad \text{ for all } x \in G.
  \end{align*}
\end{proposition}

\subsection{Stratified Banach-Lie groups}\label{sec:Banach}

We first define a \emph{Banach-Lie algebra} as a separable Banach space $X$ equipped with a continuous, bilinear, and skew-symmetric mapping $[\cdot,\cdot] : X \times X \to X$ satisfying the Jacobi identity.  We call $[\cdot,\cdot]$ the \emph{Lie bracket} of $X$.  A Banach-Lie algebra is \emph{nilpotent} if there exists some $s \in \N$ for which
\begin{align}
  [\cdots[[x_1,x_2],x_3],\cdots],x_s],x_{s+1}] = 0 , \text{ for all } x_1,\ldots,x_{s+1} \in X.
  \label{eq:s+1-brackets}
\end{align}
The smallest $s$ for which \eqref{eq:s+1-brackets} always holds is called the \emph{step of nilpotence} of $X$.  A step-$s$ nilpotent Banach-Lie algebra is \emph{graded} if there are closed subspaces $V_1,\ldots,V_s$ for which $X = V_1 \oplus \cdots \oplus V_s$ and
\begin{align*}
[x,y] \in V_{j+k}, \qquad \text{ for all } x \in V_j, y \in V_k, \text{ for all } j,k \in \N
\end{align*}
with the understanding that $V_i = \{0\}$ when $i > s$.  Finally, a graded Banach-Lie algebra $X$ is \emph{stratified} if $V_1$ generates $X$, in the sense that there is no proper closed subalgebra of $X$ containing $V_1$.

Let us fix a stratified Banach-Lie algebra $X$. One can   define a smooth group product on $X$ via the Baker-Campbell-Dynkin-Hausdorff (BCDH) formula
\begin{align} \label{eq:BCH}
  xy = x + y + \sum_{k = 2}^s \frac{(-1)^{k-1}}{k} \sum_{\begin{smallmatrix} {r_i + s_i > 0,} \\ {1 \leq i \leq k} \end{smallmatrix}} \frac{ \left( \sum_{j=1}^k (r_i + s_i) \right)^{-1}}{r_1!s_1! \cdots r_k!s_k!} \times \\
    (\ad x)^{r_1} (\ad y)^{s_1} \cdots (\ad x)^{r_k} (\ad y)^{s_k-1} y, \notag
\end{align}
where $(\ad x)y := [x,y]$.

\begin{remark} \label{rem:BCDH}
  As Lie brackets are antisymmetric, one can assume $s_k = 1$ and $r_k > 1$.  One can then also replace the last $y$ in \eqref{eq:BCH} with $x + y$.
\end{remark}

One can now define a group $G$ as $X$ endowed with the group product given by \eqref{eq:BCH}.  Such a group will be called a \emph{stratified Banach-Lie group}.

\begin{remark}
  There have been many other constructions of groups based on infinite-dimensional Lie algebras.  We briefly review and contrast with some relevant ones.  The Banach homogeneous groups studied in \cite{MagnaniRajala2014} are almost identical to stratified Banach-Lie groups except they do not require $V_1$ to generate $X$.  They can be viewed as graded Banach-Lie groups.  The infinite-dimensional Heisenberg-like groups studied in \cite{DriverGordina2008} (reviewed in Example \ref{ex.DGgroups} below) are only step 2 and the second layer $V_2$ had to be finite dimensional.  Finally, the semi-infinite Lie algebras of \cite{Melcher2009a} are nilpotent of arbitrary step but not graded and only the $V_1$ layer is allowed to be infinite-dimensional.
\end{remark}

One can endow $X$ with the dilation $\delta_\lambda$ where for $x_i \in V_i$, we have
\begin{align}
  \delta_\lambda(x_1,\ldots,x_s) = (\lambda x_1,\ldots,\lambda^s x_s), \label{eq:dilation}
\end{align}
where we use the Cartesian product on $X$ coming from the grading.
From the graded nature of $X$, one can see that this is an automorphism of $G$.

Stratified Banach-Lie groups are examples of infinite-dimensional Carnot groups.

\begin{proposition}
Stratified Banach-Lie groups can be endowed with metrics to make them infinite-dimensional Carnot groups.
\end{proposition}

\begin{proof}
Let $X$ denote the underlying Banach space of $G$.  By \cite[Lemme II.1]{Guivarch1973a} (see also \cite[Lemma 2.5]{Breuillard2014}), there exist positive constants $1 = \sigma_1,\ldots,\sigma_s$ depending on $[\cdot,\cdot]$ so that the gauge
\begin{align}
  N(x) = \max_{1 \leqslant i \leqslant  s} \sigma_i |x_i|^{1/i} \label{eq:max-norm}
\end{align}
  satisfies $N(\delta_\lambda(x)) = |\lambda|N(x)$ and $N(xy) \leqslant N(x) + N(y)$.  Although the lemmas in \cite{Guivarch1973a, Breuillard2014} are for nilpotent Lie groups, they do no make use of local compactness beyond boundedness of the Lie bracket, which is available to us.  Defining $d(x,y) := N(x^{-1}y)$, we get that $(G,d)$ is a metric scalable group.

  One can show that $G$ equipped with the topologies induced by $|\cdot|$ and $d$ are homeomorphic topological spaces.  Indeed, one immediately sees from \eqref{eq:max-norm} that origin centered balls of both $|\cdot|$ and $d$ are equivalent in terms of containment in one another.  As the Lie bracket is bounded, one gets from the BCDH formula that left translation is continuous in $|\cdot|$ (it is, of course, also continuous in $d$).  Thus, balls of $|\cdot|$ and $d$ centered around any arbitrary point are also equivalent and so the two metrics are homeomorphic.

Now let $\{x_i\}_{i \in \N}$ be a dense subset of $V_1$.  For each $n$, we let $X_n$ denote the subalgebra generated by $x_1,\ldots,x_n$.  This is a finite-dimensional stratified Lie algebra and so $X_n$ endowed with the BCDH product is a Carnot subgroup of $G$.  As $V_1$ generates $X$, we get that $\bigcup_n X_n$ is dense in $G$.

Finally, it remains to prove that $d$ is complete.  This follows immediately from the fact that on any bounded set $U \subset X$, there exist a constant $c = c_U \geqslant 1$ so that
\begin{align}
    c^{-1} |x - y| \leqslant d(x,y) \leqslant c |x-y|^{1/s}, \quad \text{for all } x,y \in U. \label{eq:holder}
\end{align}

The bound \eqref{eq:holder} essentially follows from \cite[Proposition 5.15.1]{BonfiglioliLanconelliUguzzoniBook}.  One slight change to the proof is that the product $y^{-1}x$ should be expressed in terms of the Cartesian product of the $V_k$s.  Specifically, if $x = (x_1,...,x_s)$ and $y = (y_1,...,y_s)$, then the BCDH formula allows us to write
  \begin{align*}
    (y^{-1}x)_j = x_j - y_j + \sum_{k=1}^{j-1} u_k^{(j)}
  \end{align*}
where each $u_k^{(j)} \in V_j$ is a polynomial of iterated Lie brackets of $\{x_1,...,x_{j-1},y_1,...,y_{j-1}\}$.  By Remark \ref{rem:BCDH}, one can bound $|u_k^{(j)}| \leq Q_k^{(j)}(|x|,|y|) |x_k-y_k|$ for some polynomial $Q_k^{(j)}$.  The upper bound of \eqref{eq:holder} then follows from the proof in \cite{BonfiglioliLanconelliUguzzoniBook}.

On the other hand, one can lower bound
\begin{align*}
    \left|x_j - y_j + \sum_{k=1}^{j-1} u_k^{(j)}\right| \geqslant \inf_{T_1^{(j)},...,T_{j-1}^{(j)}} \left|x_j - y_j + \sum_{k=1}^{j-1} T_k^{(j)}(x_k-y_k)\right| =: f_j(x-y),
\end{align*}
where $T_k^{(j)} : V_k \to V_j$ are operators whose norms are bounded by $Q^{(j)}_k(|x|,|y|) \leq M$ where $M$ is a constant depending only on $U$ and the Lie bracket $[\cdot,\cdot]$.  One can show that $\inf_{|z| = 1} \max_{1 \leq j \leq s} f_j(z) > 0$ by inducting on $s$ and using a triangle inequality argument.  This allows us to use the rest of the proof of \cite{BonfiglioliLanconelliUguzzoniBook} to get the lower bound of \eqref{eq:holder}.
\end{proof}

Stratified Banach-Lie groups form a proper subset of infinite-dimensional Carnot groups.  Unlike stratified Banach-Lie groups, infinite-dimensional Carnot groups need not be nilpotent \cite[Proposition 5.10]{LeDonneLiMoisala2021}.  However, even in the nilpotent case, it is not true that all infinite-dimensional Carnot groups are stratified Banach-Lie groups \cite[Proposition 5.12]{LeDonneLiMoisala2021}.

We now introduce a special family of stratified Banach-Lie groups that will be used in Section \ref{heat-kernel}.

\begin{example}[Infinite-dimensional Heisenberg-like   groups]\label{ex.DGgroups}
We recall the construction of infinite-dimensional step $2$ groups introduced in \cite{DriverGordina2008}.
These groups are modelled on an abstract Wiener space equipped with a non-degenerate continuous skew-symmetric bilinear form. We can think of the group as the Heisenberg group of an infinite-dimensional  symplectic vector space.

Suppose $W$ is a real separable Banach space $W$, and  $\mathbf{C}$ is a finite-dimensional inner product space. Define $\mathfrak{g}:=W\times \mathbf{C}$ to be an infinite-dimensional Banach-Lie algebra, which is constructed as an infinite-dimensional step $2$ nilpotent Lie algebra with continuous Lie bracket. Namely, let $\omega: W \times W \rightarrow\mathbf{C}$  be a continuous skew-symmetric bilinear form on $W$.  We will also assume that $\omega$ is surjective.

Then $\mathfrak{g}_{\omega}=\mathfrak{g}:=W\times\mathbf{C}$ is thought of as a Lie algebra with the Lie bracket given by

\begin{equation}\label{e.3.5}
[(X_1,V_1), (X_2,V_2)] := (0, \omega(X_1,X_2)).
\end{equation}
This means that $\mathfrak{g}$ is a central extension of $W$ by $\omega$. Let $G_{\omega}=G$ denote $W\times\mathbf{C}$ when thought of as a group with the multiplication law given by
\begin{equation}
  g_1 g_2 := g_1 + g_2 + \frac{1}{2}[g_1,g_2], \label{eq:DG-mult}
\end{equation}
where $g_1$ and $g_2$ are viewed as elements of $\mathfrak{g}$.  For $g_i=(w_i, c_i)$, this may be written equivalently as
\begin{equation}\label{e.3.2}
(w_1,c_1)\cdot(w_2,c_2) = \left( w_1 + w_2, c_1 + c_2 + \frac{1}{2}\omega(w_1,w_2)\right).
\end{equation}
Then $G$ is a stratified Banach-Lie group with Lie algebra $\mathfrak{g}$. These stratified Banach-Lie groups are called \emph{infinite-dimensional Heisenberg-like groups}.

Note that \eqref{eq:DG-mult} is the BCDH formula for step 2.  Thus, we have that infinite-dimensional Heisenberg-like groups are precisely step 2 stratified Banach-Lie groups with a finite-dimensional $V_2$.
\end{example}

\subsection{Convolutions and convergence of probability measures in infinite dimensions}

Below we recall several standard definitions suitable for our setting. Let $G$ be a topological group. In the case when $G$ is locally compact one can define convolutions  for Radon measures as in \eqref{def_convo}, e.g. in \cite{FollandHABook}. This definition relies on the Riesz representation theorem, so we cannot use it in our setting. Instead we use the approach in \cite{ParthasarathyBook2005} which allows us to define convolutions for probability measures on \emph{separable metric groups}. Note that in \cite{ParthasarathyBook2005} measures are probability measures, and the space of continuous functions is restricted to the space of bounded continuous functions as appropriate for the convergence of probability measures.

Suppose $G$ is a separable metric group, and denote by $\operatorname{Prob}(G)$ the space of probability measures on $G$. This choice allows us to introduce a natural topology on $\operatorname{Prob}(G)$, namely, of \emph{weak} or \emph{weak*} convergence. Note that this convergence is called weak convergence in probability literature.

\begin{dfn}[Weak convergence for probability measures]
We say that a sequence $\mu_{k}$ in $\operatorname{Prob}\left( G \right)$ \emph{converges  weakly} to a measure $\lim_{k \to \infty} \mu_{k} \in \operatorname{Prob}(G)$ if for every bounded continuous function $\phi:G \to \R$ one has
  \begin{equation}\label{def_weak_conve}
    \lim_{k\to \infty}\int_G \phi \;\dd\mu_{k} \to \int_G \phi \;\dd(\lim_{k \to \infty} \mu_{k}).
  \end{equation}
\end{dfn}
There are many equivalent ways of defining weak convergence of probability measures; we refer to the classical monograph \cite{BillingsleyConvPMBook1999}.

\begin{dfn}[Convolution of measures]
  Suppose $\mu, \nu \in \operatorname{Prob}(G)$ as follows, then

  \begin{equation}
    \left( \mu \ast \nu \right)\left( A \right) :=  \int_{G}\mu\left( A x^{-1}\right)\dd\nu \left( x \right), A \in \mathcal{B}\left( G\right),
  \end{equation}
  where $\mathcal{B}\left( G\right)$ is the Borel $\sigma$-algebra on $G$.
\end{dfn}
Consequently, for each bounded (or non-negative) Borel function $\phi:G\to \R$, one has
\begin{equation}\label{def_convo}
  \int_G \phi \;\dd(\mu \ast \nu) :=  \int_G   \int_G \phi (xy) \;\dd\mu(x) \dd\nu (y).
\end{equation}
Moreover, the convolution of probability measures is a probability measure by Fubini's theorem.

We will also need the following theorem which is essentially \cite[Theorem 1.1, p. 57]{ParthasarathyBook2005}.

\begin{theorem} \label{th:conv-cont}
Let $G$ be a separable metric group.  Then convolution is continuous with respect to the weak topology on Prob$(G)$.
\end{theorem}




\section{Quotients of scalable metric groups}\label{s.Quotients}
The following proposition concerns submetries of quotients $G\to G/N$. The statement and its proof are a direct generalization of a previous result in \cite{LeDonneRigot2016}.
\begin{proposition}\label{prop:distance_quotient}
Let $G$ be a topological group equipped with an admissible left-invariant distance function $d$. Let $N$ be a closed normal subgroup of $G$. Then the function
\begin{equation}\label{distance_quotient}
d_{G/N}(g_1 N, g_2 N):=d(g_1 N, g_2 N) :=\inf\{ d(g_1 n_1, g_2 n_2) \,:\, n_1, n_2\in N\}
\end{equation}
is an admissible left-invariant distance function on the quotient group $G/N$ for which the projection $\pi: G\to G/N$ becomes a submetry.
\end{proposition}
\begin{proof}
It is classic that $G/N$ is a (Hausdorff) topological group when $N$ is a closed normal subgroup.
We shall prove that under the assumptions on $G$ the function defined by equation \eqref{distance_quotient} metrizes $G/N$. First, observe that since $N$ is normal and the distance $d$ on $G$ is left-invariant, we have
\begin{equation}\label{distance_quotient2}
d(g_1 N, g_2 N) :=\inf\{ d(g_1  , g_2 n ) \,:\, n \in N\} .
\end{equation}
Consequently, the function $d_{G/N}$  is clearly symmetric and left-invariant. It also satisfies the triangle inequality because of \eqref{distance_quotient2}. Moreover, if $g_1 N \neq g_2 N$, i.e. if $g_2^{-1} g_1\notin N$, then, since $N$ is closed, there exists $r>0$ such that $B(  g_2^{-1} g_1,r)\cap N=\emptyset$; as standard, in a given metric space we denote by $B(g , r)$ the open ball centered at $g$ of radius $r$. Hence, $d( g_2^{-1} g_1, N)>r$ and, therefore, $d(g_1 N , g_2 N)>r\neq0$. We infer that \eqref{distance_quotient} defines a distance function on $G/H$.

Before proving that this distance is admissible we check that
\begin{equation}\label{submetry_balls}
\pi(B(g , r)) =B(\pi(g), r),\qquad \text{ for all } g\in G, \text{ for all } r\geqslant 0.
\end{equation}
Regarding the containment $\subseteq$, if $p\in \pi(B(g, r))$ with $p=\pi(\bar p)$ and $\bar p \in  B(g, r)$, then $d(\pi(\bar p), \pi(g))\leqslant d(\bar p, g)<r$ so $p\in B(\pi(g), r)$.  Regarding the containment $\supseteq$, if $p\in B(\pi(g), r)$ and $\bar p \in G$ such that $p=\pi(\bar p)$, then there exists $n\in N$ such that $d(g, \bar p n)<r$. Hence, $p=\pi(\bar p n)$ with $\bar p n \in  B(g, r)$. Thus, $p\in \pi(B(g, r))$.

Let us finally show that the distance $d_{G/N}$ defines the quotient topology on $G/N$.
Let $U\subseteq G/N$ be an open neighbourhood of some point $p\in G/N$. Thus since $\pi$ is continuous, we have that $\pi^{-1}(U) \subseteq G$ is open in $G$. Take $\bar p \in  \pi^{-1}(p)$. So there exists $r>0$ such that $B(\bar p, r) \subseteq \pi^{-1}(U)$. By \eqref{submetry_balls} we get $B(  p, r) \subseteq  U$, which then prove that $U$ is open for the topology given by the metric. Vice versa, every ball $B(  p, r) $ with $p\in G/N$ and $r\geqslant 0$ is open because of \eqref{submetry_balls} and the fact that $\pi$ is an open map.
\end{proof}

\begin{proposition}\label{quotient_by_N}
Let $(G, d, \delta)$ be a metric scalable  group  and let $N$ be a closed normal scalable subgroup. Then the quotient group $G/N$ has a natural structure of a scalable metric group  with the  dilation defined as
\begin{equation}\label{e.QuotientDilation}
\delta_\lambda(g N ) :=\delta_\lambda(g)  N,\qquad \text{ for all } g\in G,
\end{equation}
and distance given by \eqref{distance_quotient}. Moreover, if $V_1(G)$ generates $G$, then $V_1(G/N)$ generates $G/N$.
\end{proposition}

\begin{proof}
By Proposition~\ref{prop:distance_quotient} we have that the quotient topology of $G/N$ is the same as induced by the distance defined by  \eqref{distance_quotient}. We need to check that the map $\delta$ defined by \eqref{e.QuotientDilation} turns $G/N$ into a  scalable metric group. First of all, we observe that the map  $\delta: \R\times G/N\to G/N$ is well-defined.
\[
\delta_\lambda (g n )N =\delta_\lambda(g)\delta_\lambda(n)    N=\delta_\lambda(g)   N,\qquad \text{ for all } g\in G,  \text{ for all } n\in N,
\]
where we used that $\delta_\lambda $ is a group morphism and that $N$ is scalable. Moreover, such a map $\delta$ is continuous. If $\lambda_i\to \lambda$ and $g_iN \to gN$, then there exists $n_i\in N$ such that $g_in_i\to g$. Then $\delta_{\lambda_i} (g_i n_i )\to \delta_\lambda (g )$, so $\delta_{\lambda_i} (g_i n_i )N\to \delta_\lambda (g N )$, and therefore $\delta_{\lambda_i} (g_i N )\to \delta_\lambda (g  N)$.

Next we verify that  $\delta_\lambda$ is a group morphism.
\begin{multline*}
    \delta_\lambda (g_1  N) \delta_\lambda (g_2  N) =
    \delta_\lambda (g_1  )N \delta_\lambda (g_2  )N =
    \delta_\lambda (g_1  ) \delta_\lambda (g_2  )N
    \\
    =
    \delta_\lambda (g_1   g_2  )N =
    \delta_\lambda (g_1   g_2  N) =    \delta_\lambda (g_1  N g_2  N),
  \end{multline*}
  where we used that $\delta_\lambda $ is a group morphism and that $N$ is normal.
  Also, the two properties $\delta_0\equiv 1_{G/H}$ and \eqref{eq:composition_of_dilations} are obvious. Hence, $(G/N,d)$ is a scalable group.

  Finally, the scaling property of the distance is trivial.
  \begin{align*}
    d(\delta_t(g_1 N),\delta_t(g_2 N))&=
    d(\delta_t(g_1) N,\delta_t(g_2 )N)
    \\
    &=
    \inf\{ d(\delta_t(g_1) n_1, \delta_t(g_2) n_2) \,:\, n_1, n_2\in N\}
    \\
    &=
    \inf\{ d(\delta_t(g_1 n_1), \delta_t(g_2 n_2)) \,:\, n_1, n_2\in N\}
    \\
    &=
    |t|\inf\{ d(g_1 n_1, g_2 n_2) \,:\, n_1, n_2\in N\} = |t|d( g_1 N, g_2 N),
  \end{align*}
  where we used that $\delta_\lambda $ is a group morphism, that $N$ is scalable, and that the distance on $G$ has the scaling property.

The last assertion about the $V_1$s follows from the fact that the projection of the closure of the group generated by $V_1(G) $ is the closure of the group generated by $V_1(G/N)$.
\end{proof}

\begin{remark}
A scalable metric group that is abelian might not be a normed space. For example, one may consider $\R^2$ with dilations given by
  \[
    \delta_\lambda: =
    \begin{pmatrix} \lambda^2 &0\\ 0 & \lambda^2\end{pmatrix}
      \qquad\text{ or } \qquad\lambda
      \begin{pmatrix} \cos(\log\lambda) & -\sin(\log\lambda)\\ \sin(\log\lambda) & \cos(\log\lambda)\end{pmatrix}
        = \exp\left(\log(\lambda) \begin{pmatrix} 1 & -1\\ 1 & 1\end{pmatrix}\right).
        \]
See \cite{LeDonneGolo_Nicolussi2021} for the existence of scalable metrics.
\end{remark}
\begin{proposition}\label{metric_scalable_abelian}
  Let $(G,d, \delta)$ be a metric scalable   group.
  If $G$ is abelian and $V_1(G)$ generates $G$, then $G$ is a normed vector space where the scalar multiplication is given by
  \begin{equation}\label{scalar}
    t g := \delta_t g, \qquad \text{ for all } t\in \R, \text{ for all } g\in G
  \end{equation}
  and norm
  \begin{equation}\label{norm}
    \norm{g} := d(e_G, g), \qquad  \text{ for all } g\in G.
  \end{equation}
\end{proposition}
\begin{proof}
To check that the scalar multiplication is linear in the scalar we shall show that $V_1:=V_1(G)$ coincides with $G$. Let $x_1 $ and $x_2\in V_1(G)$, so that from \eqref{V1def} we have
\[
\delta_{s+t}(x_i) = \delta_{s}(x_i) \delta_{t}(x_i), \qquad \text{ for all } s,t\in \R,\text{ for } i=1,2.
\]
Hence, we have the same property for $x_1 x_2$
\begin{multline*}\delta_{s+t}(x_1 x_2) =
    \delta_{s+t}(x_1)\delta_{s+t}( x_2)
    = \delta_{s}(x_1) \delta_{t}(x_1) \delta_{s}(x_2) \delta_{t}(x_2) \\
    = \delta_{s}(x_1)  \delta_{s}(x_2) \delta_{t}(x_1) \delta_{t}(x_2)
  =\delta_{s}(x_1 x_2) \delta_{t}(x_1 x_2),
\end{multline*}
  where we used that $\delta_{s}$, $\delta_{t}$, and $\delta_{s+t}$ are group morphisms, that $x_1,x_2$ are in $ V_1$, and that $G$ is abelian.
  Moreover, because of the same properties we also have that
  \begin{multline*}\delta_{s+t}(x_1^{-1}) =
    \delta_{s+t}(x_1)^{-1} = (\delta_{s}(x_1) \delta_{t}(x_1) )^{-1} \\
  = ( \delta_{t}(x_1) \delta_{s}(x_1))^{-1} =  \delta_{s}(x_1) ^{-1} \delta_{t}(x_1)^{-1} =\delta_{s}(x_1^{-1} ) \delta_{t}(x_1^{-1} ). \end{multline*}
  Therefore, $V_1$ is a subgroup.

  As Lemma~\ref{l:V1-complete} gives that $V_1$ is closed in $G$, the assumption that $V_1$ generates $G$ implies that $V_1=G$.
  Therefore the multiplication given by $\delta$ gives a scalar multiplication \eqref{scalar} turning the abelian group $G$ into a vector space.

  Next, we show that \eqref{norm} is indeed a norm.
  We have linearity in the scalar:
  \[
    \norm{tg} = \norm{\delta_t g} = d(e_G,\delta_t g) = t d(e_G,g) = t \norm {g}.
  \]
  We have the triangle inequality:
  \[
    \norm{g_1 g_2} =   d(1,g_1 g_2) \leqslant d(1,g_1) + d(g_1,g_1 g_2)
    =d(1,g_1) + d(1,  g_2)
    =\norm{  g_1}+ \norm{g_2}.
  \]
\end{proof}

\begin{corollary}\label{c.3.5}
Let $(G,d, \delta)$ be a metric scalable group, for which $V_1(G)$ generates $G$. If $G^{\prime}$ denotes the closure of the commutator subgroup $[G,G]$, then $G/G^{\prime}$ is a normed vector space.
\end{corollary}
\begin{proof}
  Start observing that $[G,G]$ is a normal subgroup that is scalable. The closure of a normal scalable subgroup is normal and scalable. Therefore, the closure $G^{\prime}$ of the commutator subgroup $[G,G]$ is a closed, normal, and scalable subgroup. By
  Proposition~\ref{quotient_by_N} we have that naturally $G/G^{\prime}$ is a metric scalable group and, since we are assuming that $V_1(G)$ generates $G$, we also have that $V_1(G/G^{\prime})$ generates $G/G^{\prime}$.
  Moreover, since $G^{\prime}$ contains the commutator subgroup $[G,G]$, we have that the quotient $G/G^{\prime}$ is abelian. Consequently, from Proposition~\ref{metric_scalable_abelian}  we infer that
  $G/G^{\prime}$ is a normed vector space.
\end{proof}

It is not true that $[G,G]$ is always closed in an infinite-dimensional Carnot group.  Let $G := \ell_2 \times \ell_2 \times \ell_2$ and define the skew-symmetric bilinear form
\begin{align*}
  \omega : (\ell_2 \times \ell_2) \times (\ell_2 \times \ell_2) &\to \ell_2 \\
  ((x,y),(x^{\prime},y^{\prime})) &\mapsto (x_1y_1^{\prime} - x_1^{\prime}y_1, x_2y_2^{\prime} - y_2 x_2^{\prime},\ldots)
\end{align*}
By Cauchy-Schwarz and the fact that $\|\cdot\|_{\ell_2} \leqslant \|\cdot\|_{\ell_1}$, we have that $\omega$ is continuous.  We then define the group multiplication on $G$ as
\begin{align*}
  (x,y,z) \cdot (x^{\prime},y^{\prime},z^{\prime}) = (x + x^{\prime}, y + y^{\prime}, z + z^{\prime} + \frac{1}{2} \omega((x,y),(x^{\prime},y^{\prime}))).
\end{align*}
Note that this construction is similar to that of Example \ref{ex.DGgroups} except that the bilinear form takes value in an infinite-dimensional Banach space.

To show $G$ is an infinite-dimensional Carnot group, we first define the dilations
\begin{align*}
  \delta_t(x,y,z) = (tx,ty,t^2z).
\end{align*}
By \cite[Lemma 2.5]{Breuillard2014}, there exists a $\sigma > 0$ so that we can define an admissible left-invariant metric $d(g,h) = N(g^{-1}h)$ from the gauge
\begin{align*}
  N(x,y,z) = \max \{\|x\|_{\ell_2}, \|y\|_{\ell_2}, \sigma \|z\|_{\ell_2}\}.
\end{align*}
Finally, the subgroups $H_k$ generated by $x_1,\ldots,x_k,y_1,\ldots,y_k,z_1,\ldots,z_k$ are isomorphic to to the $k$-fold product of the first Heisenberg group $(\HH)^k$ and form a filtration of $G$.
We get that $[G,G] = \{(0,0)\} \times \ell_1$ which is dense but not all of $\overline{[G,G]} = \{(0,0)\} \times \ell_2$.

\begin{lemma} \label{l:V1-commutator-intersect}
  Let $G$ be a nilpotent infinite-dimensional Carnot group.  Then
  \begin{align*}
    \#(V_1(G) \cap g\overline{[G,G]}) \leqslant 1, \qquad \text{ for all } g \in G.
  \end{align*}
\end{lemma}

\begin{proof}
We will prove the lemma by induction on the step of $G$.  When $G$ is step 1, then $G$ is abelian and so the lemma follows from the fact that $\overline{[G,G]} = \{e_G\}$.

Assume now that the lemma holds for steps up to $s-1$ and let $G$ be a step $s > 1$ infinite-dimensional Carnot group. Define $G_i$ as the usual lower central series.  Thus, $G_2 = [G,G]$, $G_s \neq \{e_G\}$, and $G_{s+1} = \{e_G\}$. Let $H_{k}$ denote the filtration of $G$ by finite-dimensional Carnot subgroups, $K$ the dense subgroup $\bigcup_{k} H_{k}$, and $K_i$ denote the lower central series of $K$.  If $g \in K_s$, then $g$ is in the $s$-th lower central group of $H_j$ for some $j$.  By the finite-dimensional Carnot group theory, it follows that $g \in V_s(H_j) \subset V_s(G)$ and so $K_s \subseteq V_s(G)$.  Finally, as $K_s$ is dense in $\overline{G_s}$ and $V_s(G)$ is closed, we get that $\overline{G_s} \subseteq V_s(G)$.

We first claim that
\begin{align}
    \# (V_1(G) \cap g\overline{G_s}) \leqslant 1, \qquad \text{ for all } g \in G. \label{eq:V1-Vs-intersect}
\end{align}
Let $g \in G$ and assume $u,v \in V_1(G) \cap g\overline{G_s}$.  Then there exists $h \in \overline{G_s} \subseteq V_s(G)$ so that $uh = v \in V_1(G)$.  Let $s,t \in \R$.  As $\overline{G_s}$ lies in the center of $G$, we have
\begin{align*}
    \delta_t(uh) \delta_s(uh) = \delta_t(u)\delta_t(h) \delta_s(u)\delta_s(h) = \delta_{t+s}(u) \delta_t(h)\delta_s(h).
  \end{align*}
  On the other hand $\delta_{t+s}(uh) = \delta_{t+s}(u)\delta_{t+s}(h)$ and so we must have $h \in V_1(G)$ so that $\delta_t(h)\delta_s(h) = \delta_{t+s}(h)$.  But then $h \in V_s(G) \cap V_1(G) = \{e_G\}$ and so $u = v$, proving \eqref{eq:V1-Vs-intersect}.

  Now assume for contradiction that $u, v \in V_1(G) \cap g\overline{[G,G]}$ are distinct elements.  Define the projection $\pi : G \to G/\overline{G_s} = L$.  Then $L$ is a step $s-1$ infinite-dimensional Carnot group.  Indeed, this follows from the fact that $L$ is nilpotent and Propositions \ref{exist:filtration} and \ref{quotient_by_N}.  By \eqref{eq:V1-Vs-intersect}, we have that $\pi(u) \neq \pi(v)$.  By the construction of $\delta$ on $L$, it easily follows that $\pi(u),\pi(v) \in V_1(L)$.  Note that $\pi(g\overline{G_2}) \subseteq \pi(g)\overline{[L,L]}$.  Thus, we have that $\pi(u)$ and $\pi(v)$ are distinct elements of $V_1(L) \cap \pi(g)\overline{[L,L]}$.  But this is a contradiction of the inductive hypothesis.
\end{proof}

\section{Locally compact commutators}\label{s.LCCommutators}

As mentioned previously, an infinite-dimensional Carnot group need not be a Banach-Lie group in the sense of Section~\ref{sec:Banach}.  In this section, we will prove that this does hold when the commutator subgroup is locally compact.

\begin{theorem} \label{th:IDCG-Banach}
Let $G$ be an infinite-dimensional Carnot group.  If the commutator subgroup of $G$ is locally compact, then $G$ is a stratified Banach-Lie group.
\end{theorem}

A locally compact subgroup is automatically closed, so in contrast to the previous section, we do not need to distinguish between $[G,G]$ and its closure.

Next, we note that the hypothesis of local compactness effectively implies that $[G,G]$ is finite-dimensional.

\begin{lemma}\label{l:locally-compact-acc}
Let $G$ be a locally compact scalable group.  Suppose  $H_1 \leqslant H_2 \leqslant \dots \leqslant G$ is an increasing chain of closed scalable subgroups such that $H := \bigcup_{i=1}^\infty H_i$ is dense in $G$.  Then $G = H_N$ for some $N$.
\end{lemma}

\begin{proof}
We roughly follow the proof in \cite[I.2.4 Theorem 3]{bourbaki-evt} for topological vector spaces. Let $U \subset G$ be a relatively compact open neighborhood of the identity.  Since $H$ is dense in $G$, we have that $\{ h \delta_{1/2} (U) : h \in H\}$ is an open cover of $\overline{U}$, so there exist finitely many $h_1, \dots, h_n \in H$ so that $U \subset \overline{U} \subset \bigcup_{j=1}^n h_j \delta_{1/2} (U)$.  Since $H$ is the increasing union of the $H_i$, we have $h_1, \dots, h_n \in H_N$ for some large enough $N$.  Hence $U \subset H_N \delta_{1/2} (U)$.  Iterating, we also have $U \subset H_N \delta_{1/2}(H_N \delta_{1/2} (U)) = H_N \delta_{1/4}( U)$ because $H_N$ is a scalable subgroup.  Continuing in this way, we conclude $U \subset \bigcap_{k=1}^\infty H_N \delta_{1/2^k}( U)$, which equals $H_N$ because $H_N$ is closed.
  Since $H_N$ is scalable, we have $G = \bigcup_k \delta_k U \subset H_N$.
\end{proof}

\begin{proposition}\label{p:commutator-HN}
  Let $G$ be an infinite-dimensional Carnot group with filtration $H_1 \leqslant H_2 \leqslant \cdots$.  Suppose that $[G,G]$ is locally compact.  Then $[G,G]=[H_N, H_N]$ for some $N$.  In particular, $[G,G]$ is a finite-dimensional scalable Lie group, and $G$ is nilpotent.
\end{proposition}

\begin{proof}
  Clearly $[H_1, H_1] \leqslant [H_2, H_2] \leqslant \dots \leqslant [G,G]$ is an increasing chain of closed scalable subgroups of $[G,G]$.  To apply Lemma \ref{l:locally-compact-acc} to $[G,G]$, it only remains to show that $\bigcup_{i=1}^\infty [H_i, H_i]$ is dense in $[G,G]$.  Let $g, g' \in G$ be arbitrary and consider the commutator $[g,g'] = g^{-1} g'^{-1} g g'$.  Since $H = \bigcup_{i=1}^\infty H_i$ is dense, we can find $h_n, h_n' \in H$ with $h_n \to g$, $h_n' \to g'$.  By continuity of the multiplication and inverse operations, we have  $[h_n, h_n'] \to [g,g']$.  For some large enough $N_n$ we have $h_n, h_n' = H_{N_n}$, so $[h_n, h_n'] \in \bigcup_{i=1}^\infty [H_i, H_i]$.  Therefore we have $[g, g'] \in \overline{\bigcup_{i=1}^\infty [H_i, H_i]}$.  Since the former generate $[G,G]$ and the latter is a closed subgroup, we have $[G,G] \subset \overline{\bigcup_{i=1}^\infty [H_i, H_i]}$ as desired.  Now we invoke Lemma \ref{l:locally-compact-acc} to conclude that $[G,G] = [H_N, H_N]$ for some $N$.
\end{proof}

When $[G,G]$ is locally compact, we have the following improvement on Lemma~\ref{l:V1-commutator-intersect}.

\begin{proposition} \label{prop:V1-homeo}
  Suppose $[G,G]$ is locally compact and let $L = G / [G,G]$ equipped
  with its quotient topology (equivalently, its norm topology), with
  $\pi : G \to L$ the quotient map.  Then
  the restriction $\pi|_{V_1(G)} : V_1(G) \to L$ is a homeomorphism.
\end{proposition}

\begin{proof}
  Injectivity follows from Lemma~\ref{l:V1-commutator-intersect}.  We also have that
  $\pi|_{V_1(G)}$ has a dense image.  Indeed, let $H = \bigcup_k H_k$ denote
  the union of the finite dimensional Carnot groups from the filtration
  of $G$.  As $H$ is dense, $\pi(H)$ is dense in $L$.  We have by finite
  dimensional Carnot group theory that $\pi(H_k) = \pi(V_1(H_k))$ for
  every $k$.  Thus, $\pi(V_1(H)) = \pi(H)$ and so $\pi|_{V_1(G)}$ has
  dense image.

  We now show that the image is
  closed, so that $\pi|_{V_1(G)}$ is surjective, and that the inverse
  map $\phi = \pi|_{V_1}^{-1} : \pi(V_1(G)) \to V_1(G)$ is continuous.

  Let $y$ be an arbitrary element of $L$, so that $y = \pi(g)$ for
  some $g \in G$.  Let $y_n \in \pi(V_1(G))$ with $y_n \to y$, so that
  $y_n = \pi(x_n)$ for some (unique) $x_n \in V_1(G)$.  We want to
  show that $x_n$ converges to some $x$.  If so, then $x \in V_1(G)$,
  and by the continuity of $\pi$, we have $\pi(x_n) \to \pi(x) = y$,
  so that $y \in \pi(V_1(G))$.  Thus the image is closed and so
  $\pi|_{V_1(G)}$ is surjective.  Moreover, we have $\phi(y_n) = x_n
  \to x = \phi(y)$, and since
  $y_n$ was arbitrary, this shows that $\phi$ is continuous.

  Indeed, it suffices to show that $x_n$ has a convergent
  subsequence; this would in fact show that every subsequence of $x_n$
  has a convergent subsequence.  Moreover, every convergent
  subsequence $x_n'$ of $x_n$ must converge
  to some $x'$ satisfying $\pi(x') = y$, so by injectivity, all
  convergent subsequences converge to the same point $x$.  Thus we
  would have shown that every
  subsequence of $x_n$ has a further subsequence converging to $x$,
  which implies that $x_n \to x$.

  Using the definition of the quotient metric \eqref{distance_quotient} and
  the fact that $\pi(x_n) \to \pi(g)$, we get that there exist $u_n, v_n \in
  [G,G]$ such that $d(u_n x_n, v_n g) \to 0$.  By left-invariance, we
  have $v_n^{-1} u_n x_n \to g$, so let $k_n = v_n^{-1} u_n \in [G,G]$
  and then $k_n x_n \to g$.

  Note that, because of the scaling
  and translation invariance, if $[G,G]$ is locally compact then every
  bounded subset of $[G,G]$ is relatively compact.  Thus it suffices
  to show that $k_n$ is bounded.  If so, then by passing to a
  subsequence, we can suppose $k_n \to k$, whence $x_n = k_n^{-1} (k_n
  x_n) \to k^{-1} g$.

   Suppose, to get a contradiction, that $k_n$ is unbounded.  Passing
   to a subsequence, we can suppose $d(e, k_n) \to \infty$, so that
   $r_n := 1/d(e, k_n) \to 0$.  Then $d(e, \delta_{r_n}(k_n)) = 1$.
   Since $[G,G]$ is a scalable subgroup, we have $\delta_{r_n}(k_n)
   \in [G,G]$, so passing to a further subsequence, we can suppose that
   $\delta_{r_n}(k_n)$ converges to some $h \in [G,G]$.  Note that $d(e,
   h)=1$ and in particular $h \ne e$.  Now
  \begin{equation*}
    \delta_{r_n}(x_n) = \delta_{r_n}(k_n)^{-1} \delta_{r_n}(k_n x_n)
  \end{equation*}
  where, since $k_n x_n \to g$ and $r_n \to 0$, we have
  $\delta_{r_n}(k_n x_n) \to e$.  Hence $\delta_{r_n}(x_n) \to h^{-1}
  \ne e$.  But since $V_1(G)$ is closed and closed under scaling, we have
  $\delta_{r_n}(x_n) \in V_1(G)$ for all $n$, and hence $h^{-1} \in
  V_1(G)$.  Thus $e$ and $h^{-1}$ are distinct elements of $V_1(G)
  \cap [G,G]$, contradicting Lemma~\ref{l:V1-commutator-intersect}.
\end{proof}

The following lemma gives a slight quantitative refinement and shows that the homeomorphism $\pi|_{V_1(G)}$ preserves bounded sets.

\begin{lemma}  \label{l:V1-norm-comp}
  Let $[G,G]$ be locally compact and let $\pi : G \to (G/[G,G], \|\cdot\|)$, where $\|\cdot\|$ is the induced quotient norm.  There exists a constant $c \in (0,1)$ so that
  \begin{align}
    c d(e_G,g) \leqslant \|\pi(g)\| \leqslant d(e_G,g) \qquad \text{ for all } g \in V_1. \label{eq:V1-norm-comp}
  \end{align}
\end{lemma}

\begin{proof}
  The upper bound directly follows from the quotient metric defined in Proposition~\ref{prop:distance_quotient} (for all $g \in G$).  It remains to prove the lower bound.  By homogeneity, we may suppose $d(e_G,g) = 1$.

  Let $S$ denote all the elements $g \in V_1$ for which $d(e_G,g) = 1$.  This is a closed set.  Let $B = [G,G] \cap \overline{B(e_G,2)}$, which is a compact set.  As $B$ and $S$ do not intersect, $c:=d(B,S)  > 0$.

  Let $g \in S$.  We will show that
  \begin{align*}
    d(e_G, gh) \geqslant c, \qquad \text{ for all } h \in [G,G],
  \end{align*}
  which proves the lemma.  If $d(e_G,h) > 2$, then $d(e_G,gh) \geqslant 1$ by the triangle inequality (note that $1 \geqslant c$).  Now suppose $h \in B$, then
  \begin{align*}
    d(e_G,gh) = d(g^{-1}, h) \geqslant d(S,B) \geqslant c,
  \end{align*}
  which finishes the proof.
\end{proof}

Proposition~\ref{p:generate-Carnot} shows that every finite subset of $V_1(G)$ lies in a finite-dimensional Carnot subgroup.  When $[G,G]$ is locally compact, we get the same result for every finite subset of $G$.

\begin{lemma} \label{l:finite-contains}
  Let $[G,G]$ be locally compact.  Then for every finite subset $\{g_1,\ldots,g_n\} \subset G$, there exists a finite-dimensional Carnot subgroup $K \leqslant G$ so that $\{g_1,\ldots,g_n\} \subset K$.  Moreover, $K$ can be chosen so that $[K,K] = [G,G]$.
\end{lemma}

\begin{proof}
  As shown in Proposition \ref{p:commutator-HN}, we have $[G,G] = [H, H]$ for some (finite-dimensional) Carnot subgroup $H \leqslant G$.  Choose some finite set $y_1, \dots, y_m \in V_1(H) \subset V_1(G)$ which generates $H$ as a scalable group.

  Next, for each $1 \leqslant i \leqslant n$, let $x_i$ be the unique element of $V_1(G)$ for which $g_i \in x_i[G,G]$, whose existence is guaranteed by Proposition \ref{prop:V1-homeo}.  Let $K$ be the scalable subgroup generated by $\{x_1,\dots, x_n, y_1, \dots, y_m\}$, which by Proposition \ref{p:generate-Carnot} is a Carnot subgroup.  Then $[G,G] = [H,H] \subset [K,K]$, and it follows that $g_i \in x_i [G,G] \subset K$ for each $i$.
\end{proof}

\begin{rem}
  We do not know if local compactness of $[G,G]$ is needed for Proposition \ref{prop:V1-homeo}, Lemma \ref{l:V1-norm-comp}, or the first part of Lemma \ref{l:finite-contains} asserting the existence of a Carnot subgroup $K$ with $\{x_1, \dots, x_n\} \subset K$.
\end{rem}

Before formulating the next result, we mention that if $[G,G]$ is locally compact, then we can view $G/[G,G]$ and $[G,G]$ as linear spaces. Indeed, we can identify $G/[G,G]$ with $V_{1}(G)$ and $[G,G]$ with its Lie algebra.

We will need the following lemma.

\begin{lemma} \label{p:FD-parameterization}
  Let $G$ be a (topologically) $n$-dimensional Carnot group and $X$ be a Banach space with $\dim(X) \geqslant n$.  Suppose there are linear embeddings $f : V_1(G) \to X$ and $g : [G,G] \to X$  such that $f(X) \cap g(X) = \{0\}$.  Then there exist a subspace $Y$ of $X$, a Lie bracket law on $Y$, and a homeomorphic isomorphism $F : G \to Y$, where $Y$ is endowed with the BCDH group product, extending $f$ and $g$.  The subspace $Y$, the Lie bracket law, and the extension are all unique.

Furthermore, if $H \leqslant G$ is a Carnot subgroup, then $F|_H$ is the homeomorphic embedding extending $f|_{V_1(H)}$ and $h|_{[H,H]}$.
\end{lemma}

\begin{proof}
We identify $G$ with its Lie algebra $\g = \g_1 \oplus \cdots \oplus \g_s$ using the exponential map.  Under this identification, the set $V_1(G)$ is precisely the first layer $\g_1$ and $[G,G]$ is $\g_2 \oplus \cdots \oplus \g_s$, so the two sets are complementary subspaces of $\g$.  As such, the extension $F$ follows by simple linear algebra.  We can then pushforward the Lie bracket of $\g$ onto the image $Y$ to get the first statement.

For the second statement, it follows that if $H$ is a Carnot subgroup of $G$, then its Lie algebra $\mathfrak{h}$ is a Lie subalgebra (and thus subspace) of $\g$.  Furthermore, the Lie bracket of $\mathfrak{h}$ is the restriction of the Lie bracket of $\g$ onto $\mathfrak{h}$.
\end{proof}

\begin{lemma} \label{l:G'-homeo}
  Let $G' = [G,G]$ be locally compact and $\|\cdot\|$ any norm on $G'$ that induces its Lie algebra vector space topology.  Then the topologies on $G'$ coming from $d_G$ and $\|\cdot\|$ coincide.  Furthermore, a set $A \subset G'$ is bounded under $d_G$ if and only if it is bounded under $\|\cdot\|$.
\end{lemma}

\begin{proof}
  By Lemma \ref{l:finite-contains}, the subgroup $G'$ is the commutator subgroup of a finite-dimensional Carnot subgroup, and it is well known that the topology induced by $d_G$ and $\|\cdot\|$ are the same.  It suffices to prove that the bounded sets are also the same.

  Let $B_{G'}(r)$ and $B_{\|\cdot\|}(r)$ be balls centered at $e_G$ with respect to $d_G$ and $\|\cdot\|$, respectively.  We may assume without loss of generality that $B_{\|\cdot\|}(r)$ are axis aligned ellipsoids that are orthogonal with respect to $\{V_k(G)\}_{k=2}^s$.  We have that there exist constants $0 < r_1 < r_2 < r_3$ so that
  \begin{align}
    B_{G'}(r_1) \subseteq B_{\|\cdot\|}(r_2) \subseteq B_{G'}(r_3). \label{eq:ball-containment}
  \end{align}

  By \eqref{eq:dilation} and our assumptions of $B_{\|\cdot\|}(r)$, we have that
  \begin{align}
    \lambda B_{\|\cdot\|}(r) \subseteq \delta_\lambda(B_{\|\cdot\|}(r)) \subseteq \lambda^s B_{\|\cdot\|}(r), \qquad \text{for all } r > 0, \lambda \geqslant 1. \label{eq:convex-containment}
  \end{align}

  Let $r > 0$ and we may assume without loss of generality that $r \geqslant \max\{r_1,r_2\}$.  Then
  \begin{align*}
    B_{\|\cdot\|}(r) = \frac{r}{r_2} B_{\|\cdot\|}(r_2) \overset{\eqref{eq:convex-containment}}{\subseteq} \delta_{r/r_2}(B_{\|\cdot\|}(r_2)) \overset{\eqref{eq:ball-containment}}{\subseteq} \delta_{r/r_1}(B_{G'}(r_3)) = B_{G'}\left( \frac{r_3}{r_1}r\right).
  \end{align*}
  Similarly,
  \begin{align*}
    B_{G'}(r) = \delta_{r/r_1}(B_{G'}(r_1)) \overset{\eqref{eq:ball-containment}}{\subseteq} \delta_{r/r_1}(B_{\|\cdot\|}(r_2)) \overset{\eqref{eq:convex-containment}}{\subseteq} \left( \frac{r}{r_1} \right)^s B_{\|\cdot\|}(r_2) = B_{\|\cdot\|}\left( \frac{r_2}{r_1^s} r^s\right).
  \end{align*}
  This shows that the bounded sets of $d_G$ and $\|\cdot\|$ are the same.
\end{proof}

We can now prove Theorem~\ref{th:IDCG-Banach}.

\begin{proof}[Proof of Theorem~\ref{th:IDCG-Banach}]
  Let $G^{\prime} := [G,G]$ and $W = G/G^{\prime}$.  We let $K \subset G$ be any finite-dimensional Carnot subgroup such that $[K,K] = G^{\prime}$ as was shown to exist by Lemma \ref{l:finite-contains}.  Thus, as $[K,K]$ is graded, so is $G^{\prime}$ and we have a decomposition $G^{\prime} = \g_2 \oplus \cdots \oplus \g_s$,
where we identify $G^{\prime}$ with its Lie algebra.  We get that $W \times G^{\prime}$ is a Banach space, which we endow with the $\ell_\infty$-norm
\[
\|(u,v)\| = \max\{\|u\|_W,\|v\|_{G^{\prime}}\}.
\]
As $G^{\prime}$ is finite-dimensional, we have that  $G^{\prime}$  with the topologies induced by $d_G$ and $\|\cdot\|$ are homeomorphic topological spaces.  We will construct a Lie bracket $[\cdot,\cdot]$ on $W \times G^{\prime}$ and a map $F : G \to X$ so that when $W \times G^{\prime}$ is viewed as a Banach-Lie group, the map $F$ is an isomorphism.  Setting $W =: \g_1$, we get a grading $W \times G^{\prime} = \g_1 \oplus \g_2 \oplus \cdots \oplus \g_s$.

  We define $f : V_1(G) \to W$ to be the identification given by Proposition~\ref{prop:V1-homeo}.  We also let $g : G^{\prime} \to G^{\prime}$ be the identity.  We can then view $f$ and $g$ as maps into $W \times G^{\prime}$ in the obvious way.  We now construct a map $F : G \to W \times G^{\prime}$ as follows.  Let $x \in G$ be arbitrary.  Let $H$ be some finite-dimensional Carnot subgroup of $G$ so that $x \in H$ and $G^{\prime} \subset H$, the existence of which follows from Lemma~\ref{l:finite-contains}.  We then set $F(x) := h_H(x)$ where $h_H : H \to W \times G^{\prime}$ is the unique map extending $f|_{V_1(H)}$ and $g$ as given by Lemma~\ref{p:FD-parameterization}.

For the map $F$, injectivity is clear, while surjectivity follows from  Lemma \ref{l:finite-contains}.   To see that this is independent of choice of $H$, let $K$ be another finite-dimensional Carnot subgroup for which $G^{\prime} \subset K$ and $x \in K$.  Let $L$ be some finite-dimensional Carnot subgroup containing $K$ and $H$ and let $h_K$ and $h_L$ be the maps extending $g$ and the restriction of $f$ on the relevant subsets.  By the second part of Proposition~\ref{p:FD-parameterization}

\begin{align*}
h_H(x) = h_L(x) = h_K(x).
\end{align*}

We now define a Lie bracket on $W \times G^{\prime}$ in a similar way.  Let $x^{\prime},y^{\prime} \in W \times G^{\prime}$ and let $x,y \in G$ be their preimages under $F$.  Then $x,y$ lie in some finite-dimensional Carnot subgroup $H$ that also contains $G^{\prime}$.  We now define the Lie bracket on $h_H(H)$ to be the Lie bracket given by Proposition~\ref{p:FD-parameterization}.  That this Lie bracket is well-defined follows from reasoning similar to how we showed that $F$ is well-defined.

To see that the Lie bracket respects the grading $\g_1 \oplus \g_2 \oplus \cdots \oplus \g_s$, take two vectors $x \in \g_i, y \in \g_j$ and let $H$ be the subgroup from Lemma~\ref{l:finite-contains}.  Then, identifying $H$ with its Lie algebra, we get $H = \g_H \oplus \g_2 \oplus \cdots \oplus \g_s$ where $\g_H \subset V_1(G)$.  As $H$ is graded, we get that $[x,y] \in \g_{i+j}$, which shows that the Lie bracket on $G$ respects the grading.

  We now show $[\cdot,\cdot]$ is bounded.  First, it suffices to prove boundedness for elements of $W \times G^{\prime}$ whose image under $F^{-1}$ lies in $V = \bigcup_{k} V_{k}(G)$ as $W \times G^{\prime}$ is the direct sum of $\{f(V_{k}(G))\}_{k}$.  Let $x^{\prime} \in f(V_i(G)) \cap B_{W \times G^{\prime}}(0,1), y^{\prime} \in f(V_j(G)) \cap B_{W \times G^{\prime}}(0,1)$, $x = F^{-1}(x^{\prime})$, $y = F^{-1}(y^{\prime})$, and let $H$ be a finite-dimensional Carnot subgroup containing $G^{\prime}$ and $\{x,y\}$. By Lemmas \ref{l:V1-norm-comp} and \ref{l:G'-homeo}, we have that there is some constant $C > 0$ depending only on $G$ and the choice of norm on $G^{\prime}$ so that $x, y \in B_G(e_G,C)$.

  Applying the BCDH formula to $H$ gives that the coordinates of the commutator $xyx^{-1}y^{-1}$ in the subspace $V_{i+j}(H)$ is exactly $[x^{\prime},y^{\prime}]$.  As $x,y \in B_G(e_G,C)$, the triangle inequality gives that $xyx^{-1}y^{-1}$ lies in $G^{\prime} \cap B_G(e_G,4C)$.  By Lemma \ref{l:G'-homeo}, this is also a bounded subset under $\|\cdot\|$ and so
\begin{align*}
\sup_{x \in f(V) \cap B_{W \times G^{\prime}}(0,1)} \sup_{y \in f(V) \cap B_{W \times G^{\prime}}(0,1)} \|[x,y]\| < \infty,
\end{align*}
which proves boundedness of $[\cdot,\cdot]$.

Given the Lie bracket $[\cdot,\cdot]$, we can now define a group product on $W \times G^{\prime}$ using the BCDH product.  As $[\cdot,\cdot]$ is a bounded Lie bracket and the group is nilpotent, we have that the group product is continuous.  We also have that $F : G \to (W \times G^{\prime}, \cdot)$ is a group isomorphism, because of  Lemma \ref{l:finite-contains}. We can pushforward the metric from $G$ to a left-invariant metric $d_G$ on $W \times G^{\prime}$.  Note that $W \times \{0\}$ can be identified with $V_1(W \times G^{\prime})$.  In conclusion, on $(W \times G^{\prime}, \cdot)$ we have two topologies, one coming from $d_G$ and one coming from $\|\cdot\|$, and the group product is continuous in both topologies.  To prove the theorem, it only remains to show that these two topologies are the same.

We let $B_{W \times G^{\prime}}$ and $B_G$ denote balls in the $\|\cdot\|$ and $d_G$ metrics, respectively.  Let $(x,z) \in W \times G^{\prime}$.  To show that the topologies from $\|\cdot\|$ and $d_G$ are equivalent, we will show for every $\rho > 0$, there exists an $r > 0$ such that
\begin{align*}
B_{W \times G^{\prime}}((x,z),r) \subseteq B_G((x,z),\rho)
\end{align*}
and
\begin{align*}
B_G((x,z),r) \subseteq B_{W \times G^{\prime}}((x,z),\rho).
\end{align*}
As the group product is continuous in both topologies, it suffices to consider the case when $(x,z) = 0$.

Now fix $\rho > 0$ and let $(u,v) \in W \times G^{\prime}$ be so $\|u\|, \|v\| < r$ for some small $r > 0$ to be determined.  We have that
\begin{align*}
(u,v)(u,0)^{-1}=(u,v)(-u,0) = (0,v + P((u,v),(-u,0))).
\end{align*}
As $P$ is a polynomial of iterated Lie brackets of $[(u,v),(-u,0)]$,  we get that as $r \to 0$

\begin{align} \label{e.4.10}
  \lim_{r \to 0} \|v + P((u,v),(-u,0))\| = 0.
\end{align}
Thus
\begin{align}
  \lim_{r \to 0} d_G((u,v)(u,0)^{-1},(0,0)) = \lim_{r \to 0} d_G((0,0),(0,v+P((u,v),(-u,0)))) = 0, \label{eq:reposition-small}
\end{align}
where in the last equality, we used the fact that $G^{\prime}$ equipped with the topologies induced by $d_G$ and $\|\cdot\|$ are homeomorphic topological spaces. We now calculate

\begin{align*}
\lim_{r \to 0} d_G((u,v),(0,0)) &\leqslant \lim_{r \to 0} (d_G((u,v),(u,v)(u,0)^{-1}) + d_G((u,v)(u,0)^{-1}, (0,0)))
\\
  &\overset{\eqref{eq:reposition-small}}{\leqslant} \lim_{u \to 0} d_G((u,0),(0,0)) + 0 \\
    &= 0,
\end{align*}
where for the last equality we used Proposition \ref{prop:V1-homeo}. Thus there exists $r$ sufficiently small such that $B_{W \times G^{\prime}}((0,0),r) \subseteq B_G((0,0),\rho)$.

Now let $(u,v) \in B_G((0,0),r)$ for some small $r > 0$ to be determined.  By the definition of the norm on $W$, we have that $B_G((0,0),r) \subset B_W(0,r) \times G^{\prime}$ so that $\|u\|_W \leqslant r$.  Thus by Proposition \ref{prop:V1-homeo}
\begin{align*}
  \lim_{r \to 0} d_G((0,0),(u,0)) = 0.
\end{align*}
We then get that
\begin{align*}
  \lim_{r \to 0} d_G((u,v)(-u,0),(0,0)) \leqslant \lim_{r \to 0} d_G((u,v),(0,0)) + \lim_{r \to 0} d_G((0,0),(u,0)) = 0.
\end{align*}
Thus, as $G^{\prime}$ with respect to $d_G$ and $\|\cdot\|$ are homeomorphic topological spaces,
\begin{align*}
\lim_{r \to 0} \|(u,v)(-u,0)\| = 0.
\end{align*}
We also have that
\begin{align*}
    (u,v) - (u,v)(-u,0) = (u,P((u,v),(-u,0))).
\end{align*}
Again, similarly to \eqref{e.4.10} we get that
\begin{align*}
\lim_{r \to 0} \|(u,v) - (u,v)(-u,0)\| = 0.
\end{align*}
We now calculate
\begin{align*}
\lim_{r \to 0} \|(u,v)\| &\leqslant \lim_{r \to 0} \|(u,v) - (u,v)(-u,0)\| + \lim_{r \to 0} \|(u,v)(-u,0)\| = 0
\end{align*}
and so there exists $r$ sufficiently small so that $B_G((0,0),r) \subseteq B_{W \times G^{\prime}}((0,0),\rho)$.  This proves the theorem.
\end{proof}

\section{Convolutions of absolutely continuous measures}\label{s.CAC}
Throughout this section, we let  $G$ be an infinite-dimensional Carnot group, with first layer $V_1(G)$.
Given a locally compact subgroup $H\subset G$, we denote by $\vol_H$ any normalization of the Haar measure on $H$.
We shall say that a measure on $H$ is \emph{absolutely continuous} if it is absolutely continuous with respect to $\vol_H$, and hence absolutely continuous with respect to every other Haar measure on $H$.
 We denote by $\operatorname{AC}(H)$   the set of probability measures on $H$ that are absolutely continuous, i.e.,
\[
  \operatorname{AC}(H):=\{ \mu\in\operatorname{Meas}(H) \;:\; \mu(H)=1, \mu\ll\vol_H\}.
\]
We will also view measure in $\operatorname{AC}(H)$ as measures on $G$ supported on $H$.  We now introduce a new family of null sets based on \textbf{c}onvolutions of \textbf{a}bsolutely \textbf{c}ontinuous measures.

\begin{definition}[CAC$(\mathcal{X})$ and CAC null]\label{d.CAC}
Let $\mathcal{X} = (X_1,X_2,\ldots)$ be a Carnot-spanning set of an infinite-dimensional Carnot group $G$ with associated filtration $H_k := \langle X_1,\ldots,X_k\rangle$.
We define the set $\operatorname{CAC}^{\prime}(\mathcal{X})$ as the set of all those probability measures $\mu$ on $G$ for which for every $k \in \N$ there is a measure $\nu_{k}\in \operatorname{AC}(H_{k})$ such that for every $m\in \N$
 the sequence
  \begin{align*}
    (\nu_{k} \ast \cdots \ast \nu_m)_{k} \text{ converges weakly, as } k\to \infty,
  \end{align*}
  and the sequence
  \begin{align*}
    (\nu_{k} \ast \cdots \ast \nu_1)_{k} \text{ converges weakly to } \mu, \text{ as } k\to \infty.
  \end{align*}
  Finally, we let $\operatorname{CAC}(\mathcal{X}) := \{(L_g)_\ast\mu : g \in G, \mu \in \operatorname{ CAC}^{\prime}(\mathcal{X})\}$.
  We say that a Borel set $E \subset G$ is \emph{$\operatorname{CAC}$ null} if for every Carnot-spanning set  $\mathcal{X}$ and every $\mu \in \operatorname{CAC}(\mathcal{X})$, we have $\mu(E) = 0$.
\end{definition}

For auxiliary purposes, given a
Carnot-spanning set $\mathcal{X} = (X_1,X_2,\ldots) $   with associated filtration $H_k := \langle X_1,\ldots,X_k \rangle$,
we shall also consider the following set
\[
  \operatorname{CAC}_{k}^{\prime}(X_1,\ldots,X_{k}) := \{\nu_{k} \ast \cdots \ast \nu_1: \nu_i \in \operatorname{AC}(H_i), \text{ for all } i =1\ldots,k\}.
\]
and
$\operatorname{CAC}(X_1,\ldots,X_{k}) := \{(L_g)_\ast\mu : g \in H_{k}, \mu \in \operatorname{ CAC}^{\prime}(X_1,\ldots,X_{k})\}$.

\begin{remark}\label{r.CACversusCube}
We remark that CAC measures have some similarities to cube measures in Banach spaces.  Recall that a \emph{cube measure} in a Banach space $X$ is the distribution of a random variable of the form $a + \sum_k X_k e_k$ where $a, e_1, e_2,\ldots \in X$, the span of the $e_k$s is dense, $\sum_k \|e_k\| < \infty$, and the $X_k$ are iid uniform random variables on $[0,1]$.  If one lets $H_i$ denote the subspace generated by $e_1,\ldots,e_i$, then one sees that the cube measure is a translate of measures of the form $\cdots \ast \mu_2 \ast \mu_1$ where each $\mu_i$ is a bounded uniform probability measure on the subspace spanned by $e_i$, the complementary subgroup of $H_{i-1}$ in $H_i$.

Unlike in Banach spaces, it is not true that $H_{i-1}$ always has a complementary Carnot subgroup in $H_i$.  This is why we have to have $\nu_i$ defined on \emph{all} of $H_i$ and so $\nu_i$ will \emph{retread} $H_{i-1}$ as opposed to $\mu_i$.
\end{remark}

As in the case of cube null measures in Banach spaces, we prove that this notion of null sets is equivalent to being Aronszajn null.

\begin{theorem} \label{th:aron-cac}
A Borel set in an infinite-dimensional Carnot group with locally compact commutator subgroup is Aronszajn null if and only if it is $\operatorname{CAC}$ null.
\end{theorem}

Note that there is no assumption of local compactness on the commutator.

We will first work to prove the backward direction by showing that if a set is not Aronszajn null, then it is charged by a CAC measure.  Towards that end, let $\mathcal{X} = (X_1,X_2,\ldots)$ be a Carnot-spanning set with the associated filtration $H_k := \langle X_1,\ldots,X_k\rangle$.

For each $k \geqslant 0$ and $\eps \in(0,\infty)$,
consider some
$\nu_{k}^\eps\in \operatorname{AC}(H_{k})$.   We seek conditions for which 
\begin{align}
 e_G \text{ density point of } U\subseteq H_{k}\implies \lim_{\eps \to 0} \nu_{k}^\eps(U) = 1. \label{eq:approximation}
\end{align}
Above, and later, we shall say that a point is a \emph{density point}, if it is a density point with respect to some (and hence every) Haar measure on $H_k$.

The following lemma shows that such measures are plentiful by giving an easy way to construct such families, using the  dilations $\delta_\eps$.

\begin{lemma} \label{l:ac-aoi}
  For each $\nu \in \operatorname{AC}(H_{k})$, 
  the measures 
  $\nu_{k}^\eps := (\delta_\eps)_{\ast}(\nu)$ satisfy \eqref{eq:approximation}.
\end{lemma}

\begin{proof}
  Let $U \subseteq H_{k}$ be a set with $e_G$ as a density point and $\eta > 0$.  As $\nu$ is a probability measure, there exists $R > 0$ so that $\nu(H_{k} \backslash B_{H_{k}}(e_G,R)) < \eta/2$.  As $\nu$ is absolutely continuous, there is a $\delta > 0$ so that $\nu(A) < \eta/2$ for all sets $A \subset H_{k}$ with $\vol_{H_k}(A) < \delta$.

  As $e_G$ is a density point of $U$, there exists $\rho > 0$ so that
  \begin{align*}
    \frac{\vol_{H_k}(U \cap B_{H_{k}}(e_G,r))}{\vol_{H_k}(B_{H_{k}}(e_G,r))} > 1 - \frac{\delta}{\vol_{H_k}(B_{H_{k}}(e_G,R))}, \qquad \text{ for all } r < \rho.
  \end{align*}
  By scaling everything by $R/r$ and using homogeneity of the Haar measure on $H_{k}$, we get that
  \begin{align*}
    \vol_{H_k}(\delta_{R/r}(U) \cap B_{H_{k}}(e_G,R)) > \vol_{H_k}(B_{H_{k}}(e_G,R)) - \delta, \qquad \text{ for all } r < \rho.
  \end{align*}
  It then follows that
  \begin{align*}
    \nu(B_{H_{k}}(e_G,R) \backslash \delta_{R/r}(U)) < \eta/2, \qquad \text{ for all } r < \rho.
  \end{align*}

  Let $\eps < \rho/R$.  Then we have that
  \begin{align*}
    \nu_{k}^\eps(U) = \nu(\delta_{1/\eps}(U)) \geqslant 1 - \nu(H_{k} \backslash B_{H_{k}}(e_G,R)) - \nu(B_{H_{k}}(e_G,R) \backslash \delta_{1/\eps}(U)) \geqslant 1-\eta.
  \end{align*}
\end{proof}

Given a sequence of families of measures $\nu_1^{\eps},\nu_2^{\eps},\ldots $ with parameter $\eps\in (0,\infty)$, and a sequence of numbers $\epsilon_1,\epsilon_2,\ldots > 0$, we define for $k > m$
\begin{align*}
  \nu^{\epsilon_m,\ldots,\epsilon_k;m} := \nu_k^{\epsilon_k} \ast \cdots \ast \nu_m^{\epsilon_m}.
\end{align*}
We also write $\nu^{\epsilon_1,\ldots,\epsilon_k} = \nu^{\epsilon_1,\ldots,\epsilon_k;1}$.

The following theorem easily implies the backward direction of Theorem~\ref{th:aron-cac}.

\begin{theorem} \label{th:construction}
Let $\mathcal{X} = (X_1,X_2,\ldots)$ be a Carnot-spanning set in an infinite-dimensional Carnot group $G$,
with associated filtration $H_k := \langle X_1,\ldots,X_k\rangle$.
  For $k \in \N$ and $ \epsilon > 0$,  let $\nu_k^\epsilon\in \operatorname{AC}(H_{k})$ satisfying \eqref{eq:approximation}. Let $E \subseteq G$ be a Borel set such that $E\notin \mathcal{A}(\mathcal{X})$.  Then there exists $x \in
  E$ and a sequence $(\epsilon_1, \epsilon_2, \dots)$ such that the weak limit $\nu$  of the sequence
  $\nu^{\epsilon_1, \dots, \epsilon_n}$ exists and
  $(x\nu)(E) > 0$.
\end{theorem}

\begin{remark} \label{rem:cac-families}
  Note that this theorem is actually a little stronger than the backward direction of Theorem~\ref{th:aron-cac} as it states that we can specify the families $\{\nu_k^\epsilon\}_{k,\epsilon}$ independently of $E$.
\end{remark}


We begin with several preparatory lemmas.  We first show that the weak limit exists if the $\epsilon_i$ decay fast enough.

\begin{lemma}\label{convergent-product}
  Let $g_1, g_2, \dots$ be a sequence in $G$ with $d(e, g_i) < 2^{-i}$
  for all but finitely many $i$.  Then the infinite product $\ldots
  g_n g_{n-1} \ldots g_1$ converges in $G$.
\end{lemma}

\begin{proof}
  Let $s_n := g_n g_{n-1} \cdots g_1$.  Since the metric $d$ is
  left-invariant instead of right-invariant, we cannot easily bound
  $d(s_n, s_{n+1}) = d(s_n, g_{n+1} s_n)$.  However we can say that
  \begin{equation*}
    d(s_n^{-1}, s_{n+1}^{-1}) = d(s_n^{-1}, s_n^{-1} g_{n+1}^{-1}) =
    d(e, g_{n+1}^{-1}) = d(g_{n+1}, e) < 2^{n+1}
  \end{equation*}
  for all but finitely many $n$, and so by a triangle inequality
  argument, the sequence $s_{n}^{-1}$ is Cauchy and therefore
  converges.  Since $G$ is a topological group, inversion is
  continuous and therefore $s_n$ converges as well.
\end{proof}

\begin{lemma}\label{l:shrink-fast}
  Suppose $\nu^\epsilon_k \in \mathrm{AC}(H_k)$ is a family of
  probability measures satisfying \eqref{eq:approximation}.  There exist
  $\delta_1, \delta_2, \dots$ sufficiently small so that for all $m$
  and all sequences $\epsilon_1, \epsilon_2, \dots$ with $0 <
  \epsilon_i < \delta_i$, the sequence of measures $\nu_{m,k} =
  \nu^{\epsilon_k, \dots, \epsilon_m;m}$ converges weakly to some
  $\mu_m$ as $k \to \infty$.
\end{lemma}

\begin{proof}
  For each $i$, choose $\delta_i > 0$ sufficiently small so that $\nu_i^\epsilon(B_{H_i}(e_G,2^{-i})) > 1 - 2^{-i}$ for all $\epsilon < \delta_i$.
  Assuming $\epsilon_i < \delta_i$, let $\xi_i \sim
  \nu_i^{\epsilon_i}$ be independent $G$-valued random variables. By
  the Borel--Cantelli lemma, almost surely we have $d(e_G,\xi_i) <
  2^{-i}$ for all but finitely many $i$.  Thus for each fixed $m$, by
  Lemma~\ref{convergent-product} the sequence of elements $\sigma_k =
  \xi_k \cdots \xi_m$ converges almost surely.  Since $\sigma_k \sim
  \nu^{\epsilon_m,\dots, \epsilon_k;m}$, we have weak convergence of
  the measures.
\end{proof}






The rough idea of the proof of Theorem \ref{th:construction} is to inductively construct the sequence $\epsilon_1, \epsilon_2, \dots$, in such a way that for an appropriate $x \in E$, we have $(x \nu^{\epsilon_1, \dots, \epsilon_n})(E) > c$ for all $n$, where $0 < c < 1$ is an arbitrary fixed constant.  If $E$ were closed, it would then follow from the portmanteau theorem that $(x \nu)(E) \geqslant c$ as well, as desired.  Unfortunately, for a general Borel set $E$, we can only conclude $(x \nu)\left(\overline{E}\right) \geqslant c$, which is not good enough.

To resolve this issue, we make use of the following technical Lemma \ref{borel-decomp}.  It is mentioned in \cite[Lemma 7]{Csornyei1999a}) as a ``standard fact'' with an outline of a proof, so we fill in the details here.  Lemma \ref{borel-decomp} will allow us to shrink the set $E$ slightly at each step $n$ of the induction, producing a nested sequence of smaller Borel sets $E_n = E_{k_1, \dots, k_n}$ such that, by ensuring that $\nu^{\epsilon_1, \dots, \epsilon_n}$ puts sufficient mass on $E_n$, we can conclude that the measure $\nu$ puts sufficient mass on a closed set $\bigcap_n \overline{E_n}$ which is actually contained in $E$.

At a first reading, the reader may wish to focus on the special case of Theorem \ref{th:construction} when $E$ is closed, in which these complications do not enter.  For a closed set $E$, Lemma \ref{borel-decomp} becomes a triviality, as all the sets $E_{k_1, \dots, k_n}$ can simply be taken equal to $E$.  In that case, throughout the proof of Theorem \ref{th:construction}, all appearances of the sequence $k_1, k_2, \dots$ can be ignored.

See Section \ref{aronszajn-tight} for some further remarks on this issue.

\begin{lemma}\label{borel-decomp}
Suppose $E$ is any Borel subset of a Polish space $X$.  There exists a countable family of Borel sets $E_{k_1, \dots, k_n}$, indexed by finite sequences $(k_1, \dots, k_n)$ of natural numbers, with the  following properties:
  \begin{enumerate}
    \item $E_{\emptyset} = E$, where $\emptyset$ denotes the trivial sequence of length $n=0$; \label{decomp-trivial}
    \item   \label{decomp-nest} \label{decomp-union} For every $(k_1, \dots, k_n)$ we have
      \begin{align*}
       E_{k_1, \dots, k_n, 0} \subset E_{k_1, \dots, k_n, 1} \subset
       \dots \subset E_{k_1,\dots k_n} \text{ and } E_{k_1, \dots, k_n} = \bigcup_{j=0}^\infty E_{k_1,
      \dots, k_n, j}
      \end{align*}

    \item For every infinite sequence $k_1, k_2, \dots$, the closed set $F_{k_1,
      k_2, \dots} := \bigcap_{n=0}^\infty \overline{E_{k_1, \dots, k_n}}$ is
      contained in $E$.  \label{decomp-intersect}
  \end{enumerate}
\end{lemma}

Note that although $E_{k_1, \dots, k_n} \subset E$, the construction does not necessarily ensure that $\overline{E_{k_1, \dots, k_n}} \subset E$, so item \ref{decomp-intersect} of the lemma is nontrivial.

\begin{proof}
  Consider the Baire space $\mathbb{N}^{\mathbb{N}}$ with its product topology, whose elements we write as functions $\alpha : \mathbb{N} \to \mathbb{N}$.  There is a closed set $C \subset \mathbb{N}^{\mathbb{N}}$ and a continuous injective map $\phi : C \to X$ whose image is $E$ (\emph{c.f.} \cite[Theorem 3.3.17]{SrivastavaBook1998}).  For each finite sequence $k_1, \dots, k_n$, let $Z_{k_1, \dots, k_n} = \{ \alpha \in \mathbb{N}^{\mathbb{N}} : \alpha(i) \leqslant k_i, i=1,\dots, n\}$, which is a closed subset of $\mathbb{N}^{\mathbb{N}}$, and let $E_{k_1, \dots, k_n} = \phi(C \cap Z_{k_1, \dots, k_n})$.  Then properties \ref{decomp-trivial} and \ref{decomp-nest} follow immediately from the corresponding obvious properties of $Z_{k_1, \dots, k_n}$.

  For property \ref{decomp-intersect}, fix an infinite sequence $k_1, k_2, \dots$ and let $x \in F_{k_1, k_2, \dots}$. Then for each $n$ there is a point $x_{n} \in E_{k_1, \dots, k_n}$ with $d(x_n, x) < 1/n$, so that $x_n \to x$, and we can write $x_n = \phi(\alpha_n)$ for some $\alpha_n \in C \cap Z_{k_1, \dots, k_n}$.

  We claim that $\alpha_n$ has a convergent subsequence.  Indeed, let $\beta_n \in \mathbb{N}^{\mathbb{N}}$ be defined by
  \begin{align*}
    \beta_n(i) =
  \begin{cases}
    \alpha_n(i), & i \leqslant n \\
    0, & i > n.
  \end{cases}
  \end{align*}
  Then every $\beta_n$ is in the set $\bigcap_n Z_{k_1, \dots, k_n} = \prod_n \{0, 1, \dots, k_n\}$ which is compact by Tikhonov's theorem.  So $\beta_n$ has a convergent subsequence $\beta_{n_j} \to \beta$, that is, $\beta_{n_j}(i) \to \beta(i)$ for every $i$.  And for each $i$ we have $\alpha_n(i) = \beta_n(i)$ for all $n \geqslant i$, so $\alpha_{n_j}(i) \to \beta(i)$ as well.  Thus $\alpha_{n_j} \to \beta$.

  Since $C$ is closed, we have $\beta \in C$.  So by continuity of  $\phi$, we have $x_{n_j} = \phi(\alpha_{n_j}) \to \phi(\beta)$, hence $x = \phi(\beta) \in E$.
\end{proof}

We are now ready to prove Theorem \ref{th:construction}.

\begin{proof}[Proof of Theorem \ref{th:construction}]
  As in the theorem statement, let $\mathcal{X} = (X_1, X_2, \dots)$ be a Carnot-spanning set, $H_1, H_2, \dots$ the associated filtration, and $\nu_k^\epsilon$ the given families of absolutely continuous measures on $H_k$, satisfying \eqref{eq:approximation}.  It will be convenient to have a degenerate base case, so let us also set $H_0 = \{e_G\}$ to be the trivial subgroup, and $\nu_0 = \delta_{e_G}$ to be the unit point mass measure at the identity, which we note is the identity of the convolution operation (so $\mu \ast \delta_{e_G} = \mu$ for all probability measures $\mu$).

Let $\nu^{\epsilon_1, \dots, \epsilon_n} = \nu_n^{\epsilon_n} \ast \dots \ast \nu_1^{\epsilon_1}$, with $\nu^{\emptyset} = \nu_0 = \delta_{e_G}$, where as before $\emptyset$ denotes the trivial sequence of length 0.  Let $\delta_k$ be as in Lemma \ref{l:shrink-fast}.  From now on, each $\epsilon_i$ will denote a \emph{rational} number in $(0, \delta_i)$, so that for any such sequence of $\epsilon_i$, the sequence $\nu^{\epsilon_1, \dots,\epsilon_n}$ will converge weakly.

  Suppose $E$ is a Borel set with $E \notin \mathcal{A}(\mathcal{X})$, and let $E_{k_1, \dots, k_n}$ be a family of sets with the properties described in Lemma \ref{borel-decomp}.  Fix an arbitrary constant $0 < c < 1$.  Our goal will be to produce an $x \in G$ and  sequences $\epsilon_1, \epsilon_2, \dots$ and $k_1, k_2, \dots$ such that for every $n \geqslant 0$, we have
  \begin{equation}\label{good-condition}
    (x \nu^{\epsilon_1, \dots, \epsilon_n})(E_{k_1, \dots, k_n}) > c.
  \end{equation}
  Then, using properties of the family $E_{k_1, \dots, k_n}$ and of weak convergence, we will be able to conclude that $(x \nu)(E) \geqslant c > 0$.

  In fact, we want to find an $x$ for which we can construct the sequences $\epsilon_1, \epsilon_2, \dots$ and $k_1, k_2, \dots$ inductively, in tandem.  As the base case, note that when $n=0$, the inequality \eqref{good-condition} reads $(x \delta_{e_G})(E) > c$, which holds trivially for every $x \in E$.  So given an $x \in E$, it would suffice to show the inductive step: that for every $n \geqslant 0$, for all sequences $\epsilon_1, \dots, \epsilon_n$ and $k_1, \dots, k_n$ satisfying \eqref{good-condition}, there exist $\epsilon_{n+1}$ and $k_{n+1}$ such that \eqref{good-condition} still holds for the extended sequences $\epsilon_1, \dots, \epsilon_{n+1}$ and $k_1, \dots, k_{n+1}$.  We will say that a point $x \in E$ is ``good'' if this holds.

  To show that good points exist, we consider the complementary set $B \subset E$ of ``bad'' points and show that $B \in \mathcal{A}(\mathcal{X})$, so that in particular $B \ne E$.  Translating the above conditions into sets, we write
  \begin{align}
    A^{\epsilon_1, \dots, \epsilon_n}_{k_1, \dots, k_n} &= \left\{ x : (x \nu^{\epsilon_1, \dots, \epsilon_n})(E_{k_1, \dots, k_n}) > c \right\} \\
    B^{\epsilon_1, \dots, \epsilon_n}_{k_1, \dots, k_n} &=     A^{\epsilon_1, \dots, \epsilon_n}_{k_1, \dots, k_n} \setminus \bigcup_{k_{n+1}} \bigcup_{\epsilon_{n+1}} A^{\epsilon_1, \dots, \epsilon_{n+1}}_{k_1, \dots, k_{n+1}} \\
    B &= \bigcup_{n=0}^\infty \bigcup_{k_1, \dots, k_n} \bigcup_{\epsilon_1, \dots, \epsilon_n} B^{\epsilon_1, \dots, \epsilon_n}_{k_1, \dots, k_n} \label{B-def}
  \end{align}
  Here, a union over $k_i$ is understood to be taken from $0$ to $\infty$, and a union over $\epsilon_i$ is taken over the countable set $(0, \delta_i) \cap \mathbb{Q}$.  These are all Borel sets.  Indeed, for any Borel probability measure $\mu$ and any Borel set $D$, the function $f(x) = (x \mu)(D) = \int_G 1_D(xy) \,\mu(dy)$ is a Borel function by Fubini's theorem, as $1_D(xy)$ is a Borel function of $(x,y)$.  Then each $A^{\epsilon_1, \dots, \epsilon_n}_{k_1, \dots, k_n}$ is a set of the form $\{x : (x \mu)(D) > c\} = \{x : f(x) > c\}$ and so is Borel.  Taking countable unions and complements, $B^{\epsilon_1, \dots, \epsilon_n}_{k_1, \dots, k_n}, B$ are Borel as well.

  The heart of the proof is then the following claim:
  \begin{claim}\label{key-claim}
  $B^{\epsilon_1,\ldots,\epsilon_n}_{k_1,\ldots,k_n} \in \mathcal{A}(X_1,\ldots,X_{n+1})$.
  \end{claim}

  For brevity, let us write $B^{\epsilon_1,\ldots,\epsilon_n}_{k_1,\ldots,k_n}$ as $B_n$, and $E_{k_1, \dots, k_n}$ as $E_n$.  By Proposition \ref{th:ldlm}, to establish Claim \ref{key-claim} it suffices to show that $\vol_{H_{n+1}} (y^{-1} B_n \cap H_{n+1}) = 0$ for all $y \in G$, which we do by showing that $y^{-1} B_n \cap H_{n+1}$ contains no density points.

Suppose, to the contrary, that $x \in y^{-1} B_n \cap H_{n+1}$ is a density point.  Multiplying on the left by $x^{-1}$, we have that $e_G$ is a density point of the set $D = x^{-1} y^{-1} B_n \cap H_{n+1}$. (Since $x \in H_{n+1}$, we have $x^{-1} H_{n+1} = H_{n+1}$).  Since $yx \in B_n$, for all $\epsilon_{n+1}$ and $k_{n+1}$ we have $yx \notin A^{\epsilon_1, \dots, \epsilon_{n+1}}_{k_1, \dots, k_{n+1}}$, i.e.:
  \begin{equation*}
    (yx \nu^{\epsilon_1, \dots, \epsilon_{n+1}})(E_{k_1, \dots, k_{n+1}}) \leqslant c.
  \end{equation*}
  But $E_n = E_{k_1, \dots, k_n} = \bigcup_{k_{n+1}=0}^\infty E_{k_1, \dots, k_{n+1}}$ as an increasing union, so it follows that
  \begin{equation}\label{measure-small}
    (yx \nu^{\epsilon_1, \dots, \epsilon_{n+1}})(E_n) \leqslant c.
  \end{equation}
  On the other hand, since $\nu^{\epsilon_1, \dots, \epsilon_{n+1}} = \nu_{n+1}^{\epsilon_{n+1}} \ast \nu^{\epsilon_1, \dots, \epsilon_n}$, we can write
  \begin{align*}
    (yx \nu^{\epsilon_1, \dots, \epsilon_{n+1}})(E_n) = \int_{H_{n+1}} (yxz\nu^{\epsilon_1, \dots, \epsilon_n})(E_n) \,\nu^{\epsilon_{n+1}}_{n+1}(dz).
  \end{align*}
  Now if $z \in D$, then $yxz \in B_n \subset A^{\epsilon_1, \dots, \epsilon_n}_{k_1, \dots, k_n}$ and thus $(yxz\nu^{\epsilon_1, \dots, \epsilon_n})(E_n) > c$.  Hence
  \begin{equation}\label{measure-large}
    (yx \nu^{\epsilon_1, \dots, \epsilon_{n+1}})(E_n) > c \cdot \nu^{\epsilon_{n+1}}_{n+1}(D).
  \end{equation}
  But as $\nu^{\epsilon}_{n+1}$ satisfies \eqref{eq:approximation}, we have $\nu^{\epsilon_{n+1}}_{n+1}(D) \to 1$ as $\epsilon_{n+1} \to 0$.  So for small enough $\epsilon_{n+1}$, Eqs.~\eqref{measure-small} and \eqref{measure-large} contradict each other.  Claim \ref{key-claim} is proved.

  We can now proceed following the logic sketched above.  Since $B^{\epsilon_1,\ldots,\epsilon_n}_{k_1,\ldots,k_n} \in \mathcal{A}(X_1,\ldots,X_{n+1})$, we can write $B^{\epsilon_1,\ldots,\epsilon_n}_{k_1,\ldots,k_n} = \bigcup_{i = 1}^{n+1} B^{\epsilon_1,\ldots,\epsilon_n}_{k_1,\ldots,k_n}(i)$ with $B^{\epsilon_1,\ldots,\epsilon_n}_{k_1,\ldots,k_n}(i) \in \mathcal{A}(X_i)$.  Then \eqref{B-def} can be re-indexed as
  \begin{equation*}
        B = \bigcup_{i=1}^\infty \left[ \bigcup_{n=i-1}^\infty \bigcup_{k_1, \dots, k_n} \bigcup_{\epsilon_1, \dots, \epsilon_n} B^{\epsilon_1, \dots, \epsilon_n}_{k_1, \dots, k_n}(i) \right]
  \end{equation*}
  Since $\mathcal{A}(X_i)$ is closed under countable unions, the bracketed set is in $\mathcal{A}(X_i)$, so we have $B \in \mathcal{A}(\mathcal{X})$.    Thus there exists some ``good point'' $x \in E \setminus B$.  We construct, by induction, sequences $\epsilon_1, \epsilon_2, \dots$ and $k_1, k_2, \dots$ for which \eqref{good-condition} holds for all $n$.  As noted above, \eqref{good-condition} holds with $n=0$ simply because $x \in E$.  Suppose now that $\epsilon_1, \dots, \epsilon_n$ and $k_1, \dots, k_n$ have been constructed and that they satisfy \eqref{good-condition}.  This means that $x \in A^{\epsilon_1, \dots, \epsilon_n}_{k_1, \dots, k_n}$.  Since $x \in E \setminus B$, we have $x \notin B^{\epsilon_1, \dots, \epsilon_n}_{k_1, \dots, k_n}$, and therefore there must exist $\epsilon_{n+1}, k_{n+1}$ so that $x \in A^{\epsilon_1, \dots, \epsilon_{n+1}}_{k_1, \dots, k_{n+1}}$, which is to say that $\epsilon_1, \dots, \epsilon_{n+1}$ and $k_1, \dots, k_{n+1}$ satisfy \eqref{good-condition}.  By induction, we have shown that with the constructed sequences, \eqref{good-condition} holds for all $n$.

  To finish the proof, let $\nu$ be the weak limit of $\nu^{\epsilon_1, \dots, \epsilon_n}$, which exists because every $\epsilon_i$ was chosen less than $\delta_i$.  Now for any $n > m$ we have
  \begin{equation*}
    (x\nu)^{\epsilon_1, \dots, \epsilon_n}\left(\overline{E_{k_1, \dots, k_m}}\right) \geqslant (x\nu)^{\epsilon_1, \dots, \epsilon_n}\left(E_{k_1, \dots, k_m}\right) \geqslant (x\nu)^{\epsilon_1, \dots, \epsilon_n}\left(E_{k_1, \dots, k_n}\right) > c
  \end{equation*}
  and so letting $n \to \infty$ and using the portmanteau theorem, we have $(x\nu)\left(\overline{E_{k_1, \dots, k_m}}\right) \geqslant c$.  Now we have $\bigcap_{m=1}^\infty \overline{E_{k_1, \dots, k_m}} \subset E$ as a decreasing countable intersection, so using continuity from below we have
  \begin{equation*}
    (x\nu)(E) \geqslant (x\nu)\left(\bigcap_{m=1}^\infty \overline{E_{k_1, \dots, k_m}}\right) = \lim_{m \to \infty} (x\nu)\left(\overline{E_{k_1, \dots, k_m}}\right) \geqslant c
  \end{equation*}
  and the proof of Theorem \ref{th:construction} is complete.
\end{proof}

We now prove some lemmas that will be useful for proving the forward direction of Theorem \ref{th:aron-cac}.

\begin{lemma} \label{l:cack-ac}Elements in
  $\operatorname{CAC}_{k}(X_1,\ldots,X_{k})$ are absolutely continuous measures  on $H_{k}$.
\end{lemma}

\begin{proof}
  Take $g \in H_{k}$ and $\nu_i \in \operatorname{AC}(H_i), \text{ for each } i =1,\ldots,k$.
  We need to prove that $(L_g)_\ast(\nu_{k} \ast \cdots \ast \nu_1) \ll\vol_{H_{k}}$.
  Take a set $E\subset H_{k} $ such that  $\vol_{H_{k}}(E)=0$.
  Then first, since $H_{k}$ is nilpotent (and hence unimodular) we have that  $\vol_{H_{k}}$ is right and left  translation invariant, thus
  $\vol_{ H_{k}}(R_b(L_a(E)))=0$ for all $a,b\in H_{k}$.
  Second,
  recall  that $\nu_{k}\ll\vol_{H_{k}}$, and therefore
  \begin{equation}\label{ciao}\nu_{k}(R_b(L_a(E)))=0, \qquad \text{ for every } a,b\in H_{k}.
  \end{equation} Thus
  recalling that   the measure  $\nu:=\nu_{k-1}\ast \cdots \ast \nu_1  $ is   supported on $H_{k}$,
  we have
  \begin{align*}(L_g)_\ast(\nu_{k} \ast \cdots \ast \nu_1)(E) &=
    (\nu_{k} \ast   \nu )(g^{-1}E)\\
    & =   \int_G  \uno_{g^{-1}E}  \;\dd(\nu_{k} \ast\nu)
    \\
    &\overset{\eqref{def_convo}}{=}
    \int_G   \int_G \uno_{g^{-1}E}  (xy) \;\dd\nu_{k}(x) \dd\nu (y) \\
    \\
    &=
    \int_G   \int_G \uno_{g^{-1}Ey^{-1}}  (x) \;\dd\nu_{k}(x) \dd\nu (y)
    \\
    &=
    \int_G     \nu_{k}({g^{-1}Ey^{-1}}  ) \;  \dd\nu (y)
    \\&=
    \int_{H_{k}}     \nu_{k}(R_{y^{-1}}(L_{g^{-1}}(E)  )) \;  \dd\nu (y)
    \\&\overset{\eqref{ciao}}{=}
    \int_{H_{k}}     0 \;  \dd\nu (y)=0
    .
  \end{align*}
\end{proof}

\begin{lemma} \label{l:cac-decomp}
  For every $k \geqslant 1$ and $\mu \in \operatorname{CAC}^{\prime}(\mathcal{X})$, there exist measures $\nu \in \operatorname{CAC}_{k}^{\prime}(X_1,\ldots,X_{k})$ and $\eta$ on $G$ for which
  \begin{align*}
    \mu = \eta \ast \nu.
  \end{align*}
\end{lemma}

\begin{proof}
  Let $\mu$ be of the form $\lim_j (\nu_j \ast \cdots \ast \nu_1)$ with $\nu_j\in \operatorname{ AC}(H_j)$  for $j \in \N$ and for all $m\in \N$ the weak limit $\lim_j
  (\nu_j \ast \cdots \ast \nu_m)  $ exists. Fixed $k\in \N$, take
  $\eta:= \lim_j
  (\nu_j \ast \cdots \ast \nu_{k})$ and
  $\nu := \nu_k \ast \cdots \ast \nu_1$.  Because of Theorem~\ref{th:conv-cont}, which we can apply since $\nu$ and each $(\nu_j \ast \cdots \ast \nu_{k})$ are   probability measures,
  we get the lemma.
\end{proof}

We can now finally prove Theorem \ref{th:aron-cac}.

\begin{proof}[Proof of Theorem~\ref{th:aron-cac}]
  The backward direction follows from Theorem \ref{th:construction}.  We now prove the forward direction.

 Given an Aronszajn null set $E$ and a Carnot-spanning set $\mathcal{X} = (X_1,X_2,\ldots)$,
 we shall first prove  that
\begin{align}
\mu(E) = 0, \qquad \text{ for all } \mu \in \operatorname{CAC}^{\prime}(\mathcal{X}). \label{eq:CAC' null}
\end{align}

Note that $E \in \cA(\mathcal{X})$ by assumption.  Thus, $E = \bigcup_{k=1}^\infty E_{k}$ where $E_{k} \in \cA(X_{k})$.  It suffices to show that $\mu(E_{k}) = 0$ for each $k$.

  Fixed $k$, let $\nu \in \operatorname{CAC}_{k}^{\prime}(X_1,\ldots,X_{k})$ and $\eta$ be the measures from applying Lemma~\ref{l:cac-decomp} to $\mu$.  Then
  \begin{align*}
    \mu(E) = \int_G \int_{H_{k}} {\bf 1}_E(xy) ~d\nu(y) ~d\eta(x) = \int_G \nu(x^{-1}E) ~d\eta(x).
  \end{align*}
  As $x^{-1}E \in \cA(X_{k})$ for every $x \in G$, we get from Proposition~\ref{th:ldlm} and Lemma~\ref{l:cack-ac} that $\nu(x^{-1}E) = 0$, which gives that $\mu(E) = 0$.

  Now let $\nu \in \operatorname{CAC}(\mathcal{X})$ for some Carnot-spanning set $\mathcal{X}$ and let $E$ be Aronszajn null.  Then there is some $g \in G$ and $\mu \in \operatorname{ CAC}^{\prime}(\mathcal{X})$ so that
  \begin{align*}
    \nu(E) = \mu(gE) = 0,
  \end{align*}
  where in the final equality, we used \eqref{eq:CAC' null} and the fact that $gE$ is the left-translate of an Aronszajn null set and so is Aronszajn null.
\end{proof}

\section{Gauss--Haar measures}\label{s.HM}

The original inspiration for the present article was the result of Cs\"ornyei \cite{Csornyei1999a} which showed that in a Banach space, a Borel set $B$ is Aronszajn null if and only if it is Gauss null; that is, if $\gamma(B) = 0$ for every non-degenerate Gaussian measure $\gamma$ on $B$.  Our main
Theorem~\ref{th:aron-cac} can be seen as an analogue of Cs\"{o}rnyei's theorem in the setting of infinite-dimensional Carnot groups, in which the role of
Gaussian measures is played by CAC measures.

In this and the next section, we specialize to the case of an infinite-dimensional Carnot group $G$ with a locally compact commutator subgroup.  By Theorem \ref{th:IDCG-Banach}, $G$ is in fact a stratified Banach-Lie group of some finite step $s$, so that as a set $G$ is equal to a Banach Lie algebra $(\mathfrak{g}, [\cdot, \cdot])$, equipped with the group operation defined by the BCDH formula \eqref{eq:BCH}.  The exponential map in this case is the identity map, and we shall not write it when passing between the group and Lie algebra interpretations of this space.

Let $\mathfrak{g} = V_1 \oplus \dots \oplus V_s$ be the stratification of $\mathfrak{g}$.  To match the notation of Example \ref{ex.DGgroups}, let $W = V_1$ and $\mathbf{C} = V_2 \oplus \dots \oplus V_s$, so that $\mathbf{C}$ is a finite-dimensional nilpotent Lie algebra.  With respect to the scalable group structure on $G$, we have $W = V_1(G)$ and $\mathbf{C} = [G,G]$.  Let $m$ denote Lebesgue measure on $\mathbf{C}$ (in any fixed normalization), which is also the Haar measure on the subgroup  $\mathbf{C} = [G,G]\subset G$.

\begin{lemma}\label{exp-invariant}
  Fix $x \in W$.  The map $\psi_x : \mathbf{C} \to \mathbf{C}$ defined by $\psi_x(y) = xy - x$ is a bijection that leaves the Lebesgue measure $m$ invariant.
\end{lemma}

\begin{proof}
  It is clear that $\psi_x$ is a bijection, with inverse $\psi_x^{-1}(z) = x^{-1}(z+x)$.  To see it is measure preserving, for each $y \in \mathbf{C}$ let us write $y = y_2 + \dots + y_s$ where $y_k \in V_k$.
  If we write out $\psi_x(y)$ using the BCDH formula \eqref{eq:BCH}, then noting that $\ad x, \ad y$ map $V_k$ into $V_{k+1} \oplus \dots \oplus V_s$,
   we see that for each $k$ the $V_k$ component of $\psi_x(y)$ is given by $y_k$ plus a term depending only on $y_2, \dots, y_{k-1}$.  So if we fix a basis for $C$ adapted to the stratification, the Jacobian matrix of $\psi_x$ is lower triangular with 1s on the diagonal, and so its determinant is 1.
\end{proof}

\begin{corollary}\label{exp-integral}
  Let $\gamma$ be an arbitrary probability measure on $W$.  Then for any Borel function $f$, we have
  \begin{align*}
    \int_W \int_{\mathbf{C}} f(xy)\, m(dy)\,\gamma(dx) &=     \int_W \int_{\mathbf{C}} f(x + \psi_x(y) )\, m(dy)\,\gamma(dx)  \\
    &= \int_W \int_{\mathbf{C}} f(x+y)\, m(dy)\,\gamma(dx)
  \end{align*}
\end{corollary}

We will be interested in the case where $\gamma$ is a Gaussian measure, so we now recall some facts about Gaussian measures on Banach spaces; we refer the reader to \cite{BogachevGaussianMeasures} for full details.  Here we consider only \emph{centered} Gaussian measures (i.e. having mean zero).

A Borel probability measure $\gamma$ on a separable real Banach space $W$ is said to be a \emph{Gaussian measure} if for every continuous linear functional $f \in W^{\ast}$, the pushforward $f_{\ast} \gamma$ is a mean-zero Gaussian measure on $\mathbb{R}$.  Restating this definition in terms of characteristic functions, $\gamma$ is Gaussian if and only if for every $f \in W^{\ast}$ we have $\int_W e^{i f(x)}\,\gamma(dx) = e^{-q(f,f)/2}$, where $q$ is a continuous, positive semi-definite, quadratic form on $W^{\ast}$, called the \emph{covariance form} of $\gamma$.

We say a Gaussian measure $\gamma$ is \emph{non-degenerate} if the following equivalent conditions hold.
\begin{itemize}
\item For all nonzero $f \in W^{\ast}$, the pushforward $f_{\ast} \gamma$ is a non-degenerate Gaussian measure on $\mathbb{R}$ (i.e. it is not the Dirac measure $\delta_0$, which we can view as a Gaussian measure of variance zero);
\item The covariance form $q$ is positive-definite;
\item The topological support of $\gamma$ is all of $W$.
\end{itemize}
In this case, the covariance form $q$ can be viewed as an inner product on $W^{\ast}$ (typically incomplete).  Now it can  be shown that there is an injective linear map $J: W^{\ast} \to W$ such that $q(f,g) = f(Jg)$ for all $f,g \in W^{\ast}$.  If we push forward the inner product $q$ onto $\operatorname{ran} J \subset W$, and then complete $\operatorname{ran} J$ under this inner product, the resulting Hilbert space $(H, \langle \cdot, \cdot \rangle_H)$ is called the \emph{Cameron--Martin space} of $\gamma$.  It turns out that the inclusion map $\operatorname{ran} J \hookrightarrow W$ extends injectively to $H$, so that $H$ can be identified with a linear subspace $H \subset W$ which, under these assumptions, is dense in the $W$-norm.  The $H$-inner product is stronger than the $W$-norm, so in particular, any dense subset $E \subset H$ (with respect to the $H$-inner product) is dense in $W$ with respect to the $W$-norm.  The triple $(W,H,\gamma)$ is also known as an \emph{abstract Wiener space}, e.g. \cite{Gross1967b}.

We recall the celebrated Cameron--Martin quasi-invariance theorem.
\begin{theorem}\label{cameron-martin}
A Gaussian measure $\gamma$ is \emph{quasi-invariant} under translation by elements of its Cameron--Martin space $H$.  That is, if for $y \in W$ we consider the translation map $L_y$ on $W$ defined by $L_y(x) = x+y$, then the translated measure $\gamma_y = (L_y)_{\ast} \gamma$ is mutually absolutely continuous with $\gamma$ when $y \in H$, and when $y \notin H$ we have $\gamma_y \perp \gamma$.
\end{theorem}

We also recall for later use the following properties of Gaussian measures.
\begin{itemize}
\item A dilation of a Gaussian measure is a Gaussian measure;
\item The convolution of two Gaussian measures is a Gaussian measure;
\item The weak limit of a sequence of Gaussian measures, if it exists, is a Gaussian measure.
\end{itemize}

So following the product or convolution constructions of the previous subsection, we obtain a natural class of measures and null sets for study.

\begin{definition}\label{def:gauss-haar}
By a \emph{Gauss--Haar measure} on $G = W \times \mathbf{C}$ we mean a measure $\nu$ of the form $\nu = \gamma \times m$ where  $\gamma$ is a non-degenerate Gaussian measure on $W$ and $m$ is Lebesgue measure on $C$.  We will say a Borel set $E \subset G$ is \emph{Gauss--Haar null} if for every Gauss--Haar measure $\nu$
and every $g \in G$, we have $(L_g)_\ast \nu(E) = \nu(g^{-1} E) = 0$.
\end{definition}

The goal of this section is to prove the following theorem.

\begin{theorem}\label{A-null-iff-GH-null}
A Borel set $E \subset G$ is Aronszajn null if and only if it is Gauss--Haar null.
\end{theorem}

\subsection{Quasi-invariance}

The concept of quasi-invariance that appears in the Cameron--Martin theorem (Theorem \ref{cameron-martin}) has a natural extension to measures on topological groups. For a more general approach to measurable group actions and quasi-invariance we refer to \cite{Gordina2017}.

\begin{definition}
Let $G$ be a topological group and let $\nu$ be a measure on $G$.  As usual, let $L_g(h) = gh$ and $R_g(h) = hg$ be the left and right translation maps.  For a set $Q \subset G$, we say $\nu$ is quasi-invariant under left (respectively, right) translation by $Q$ if for every $g \in Q$, the measures $\nu$ and $(L_g)_{\ast} \nu$ (respectively, $(R_g)_{\ast} \nu$) are mutually absolutely continuous.
\end{definition}

Roughly following the idea of \cite[Proposition 5]{Phelps1978}, we have the following lemma.

\begin{lemma}\label{A-null-qi}
  Let $G$ be a nilpotent infinite-dimensional Carnot group (here $[G,G]$ does not need to be locally compact), and $\nu$ a $\sigma$-finite measure on $G$.  Suppose there exists a subset $Q \subset V_1(G)$ that is dense in $V_1(G)$ and closed under dilation, such that $\nu$ is quasi-invariant under right translation by $Q$.  Then for every Aronszajn null set $E \subset G$, we have $\nu(E)=0$.  Moreover, the same is true for every left translate $(L_g)_{\ast} \nu$, where $g \in G$.
\end{lemma}

\begin{proof}
Since $Q$ is dense in $V_1(G)$, we can find a sequence $Y_1, Y_2, \dots \in Q$ which is Carnot spanning.  For instance, let $Y_1, Y_2, \dots$ be a dense sequence in $Q$ (recall that $G$ and hence $Q$ are assumed to be separable).  By Proposition \ref{p:generate-Carnot}, for each $k$, the scalable subgroup $H_k$ generated by $Y_1, \dots, Y_k$ is finite-dimensional.  And the closed scalable subgroup $\overline{\bigcup_{k=1}^\infty H_k}$ contains every $Y_k$, hence contains $Q$, hence contains $V_1(G)$, hence contains all of $G$ by Proposition \ref{exist:filtration}.

Now suppose $E \subset G$ is Aronszajn null, so that in particular $E \in \mathcal{A}(Y_1, Y_2, \dots)$.  This means there is a decomposition $E = \bigcup_{k=1}^\infty E_k$ with $E_k \in \mathcal{A}(Y_k)$.  Let $k$ be arbitrary; it suffices to show $\nu(E_k)=0$.  By definition of $\mathcal{A}(Y_k)$, for every $g \in G$ we have
\begin{equation}\label{A-null-integral}
    \int_{\mathbb{R}} 1_{E_k}(g \delta_t(Y_k))\,dt = 0.
\end{equation}
Integrating \eqref{A-null-integral} with respect to $\nu$ and using Fubini's theorem, we see that for almost every $t \in \mathbb{R}$ we have
\begin{equation*}
    \int_G 1_{E_k}(g \delta_t(Y_k))\,\nu(dg) = [(R_{\delta_t(Y_k)})_{\ast} \nu](E_k) = 0.
\end{equation*}
Fix any such $t$.  Since $Q$ was assumed to be closed under dilation, we have that $\nu$ is quasi-invariant under right translation by $\delta_t(Y_k)$ and thus $\nu(E_k)=0$.  We conclude that $\nu(E)=0$, and the corresponding statement for $(L_g)_{\ast} \nu$ follows by applying the same argument to the set $g^{-1} E$ which is also Aronszajn null.
\end{proof}

We remark that in the case when $G$ is a Banach space (i.e. an abelian infinite-dimensional Carnot group), then combining Lemma \ref{A-null-qi} with the Cameron--Martin Theorem \ref{cameron-martin} yields a streamlined proof of Phelps' result that Aronszajn null subsets of a Banach space are Gauss null \cite{Phelps1978}.

We now assume again that $[G,G]$ is locally compact, so that $G$ is a stratified Banach-Lie group with a finite-dimensional commutator subgroup.  In order to apply Lemma \ref{A-null-qi} to a Gauss--Haar measure, we first show the following.
\begin{lemma}\label{gauss-haar-qi}
Let $\gamma$ be a non-degenerate Gaussian measure on $W$ with Cameron--Martin space $H$.  Then the Gauss--Haar measure $\gamma \times m$ is quasi-invariant under left or right translation by $H$.
\end{lemma}

\begin{proof}
Suppose $(\gamma \times m)(E) = 0$.  By Corollary \ref{exp-integral} this implies that
\begin{equation*}
    \int_W \int_{\mathbf{C}} 1_E(x y) \,m(dy)\,\gamma(dx) = 0.
\end{equation*}
Fix $h \in H$.  We then have
\begin{align*}
  [(L_{h})_{\ast} (m \times \gamma)](E) &=
  \int_W \int_{\mathbf{C}} 1_E(h x y) \,m(dy)\,\gamma(dx) \\
  &= \int_W \int_{\mathbf{C}} 1_E((h+x) (h+x)^{-1} hx y) \,m(dy)\,\gamma(dx) \\
  &= \int_W \int_{\mathbf{C}} 1_E((h+x) y) \,m(dy)\,\gamma(dx)
\end{align*}
  because $(h+x)^{-1}hx \in \mathbf{C}$ by the Baker--Campbell--Hausdorff--Dynkin formula, and the Haar measure $m$ on $\mathbf{C}$ is invariant under left translation.  But
  \begin{equation*}
    \int_W \int_{\mathbf{C}} 1_E((h+x) y) \,m(dy)\,\gamma(dx) = \int_W \int_{\mathbf{C}} 1_E(x y) \,m(dy)\,\gamma_h(dx) = 0
  \end{equation*}
  because, by the Cameron--Martin theorem, $\gamma_h$ is absolutely continuous to $\gamma$.
\end{proof}

Combining Lemmas \ref{A-null-qi} and \ref{gauss-haar-qi}, we have proved half of Theorem \ref{A-null-iff-GH-null}, namely,

\begin{corollary}
If $E \subset G$ is Aronszajn null, then it is Gauss--Haar null.
\end{corollary}

\subsection{Gauss--Haar and CAC measures}

To prove the remaining direction of Theorem \ref{A-null-iff-GH-null}, we will make use of Theorem \ref{th:construction}.  Since Gauss--Haar measures are not CAC measures (they are not even probability measures), instead of using Theorem \ref{th:construction} directly, we construct a family of CAC measures to which it can be applied, each of which will be absolutely continuous to a Gauss--Haar measure.

Let $\pi : G \to W$ denote the projection onto the first layer of the Lie algebra; i.e. such that $\pi(x+y) = x$ for $x \in W$, $y \in \mathbf{C}$.  Note that we also have $\pi(xy) = x$, so $\pi$ is a continuous and surjective scalable group homomorphism from $G$ to the Banach space $W$, whose kernel is $\mathbf{C}$.

Let $\star$ denote the convolution operation for probability measures on the Banach space $W$.  We continue to write $\ast$ for the convolution of measures on $G$.  Note that for probability measures $\nu_1, \nu_2$ on $G$, we have $\pi_{\ast}(\nu_1 \ast \nu_2) = \pi_{\ast} \nu_1 \star \pi_{\ast} \mu_2$.

\begin{lemma}\label{ac-inf}
  Let $\nu_0, \nu_1$ be two probability measures on $G$, and let $\nu = \nu_1 \ast \nu_0$.  Suppose $\nu_0$ has the property that $\nu_0 \ll (\pi_{\ast} \nu_0 \times m)$.  Then $\nu \ll (\pi_{\ast} \nu \times m)$.
\end{lemma}

\begin{proof}
  Let $E$ be a Borel set with $(\pi_{\ast} \nu \times m)(E) = 0$.  Using Corollary \ref{exp-integral},
  this implies
  \begin{align*}
    0 &= \int_G \int_{\mathbf{C}}   1_E(\pi(g) y) \,m(dy)\,\nu(dg) \\
    &= \int_G \int_G \int_{\mathbf{C}} 1_E(\pi(gh)y) \, m(dy)\, \nu_0(dh) \, \nu_1(dg).
  \end{align*}
  So if we set $F(g,h,y) = 1_E(\pi(gh)y)$, we have that $F = 0$, $\nu_1 \times \nu_0 \times m$-a.e.

  Let $\rho$ be the density of $\nu_0$ with respect to $\pi_{\ast} \nu_0 \times m$.  Then we write:
  \begin{align*}
    \nu(E) &= \int_G \int_G 1_E(gk)\,\nu_0(dk)\,\nu_1(dg) \\
    &=    \int_G \int_G \int_{\mathbf{C}} 1_E(g\pi(h) y) \rho(\pi(h), y) \,m(dy)\,\nu_0(dh)\,\nu_1(dg).
  \end{align*}
  Now let us write $g \pi(h) y = \pi(gh) \pi(gh)^{-1} g \pi(h) y$.  It is easy to verify that $\pi(gh)^{-1} g \pi(h) \in \mathbf{C}$, so we have
  \begin{align*}
    \nu(E) &= \int_G \int_G \int_{\mathbf{C}} F(g, h, \pi(gh)^{-1} g \pi(h) y) \rho(\pi(h), y) \,m(dy)\,\nu_0(dh)\,\nu_1(dg) \\
    &= \int_G \int_G \int_{\mathbf{C}} F(g, h, z) \rho(\pi(h),  \pi(h)^{-1} g^{-1} \pi(gh) z) \,m(dz)\,\nu_0(dh)\,\nu_1(dg)
  \end{align*}
  where we made the change of variables $z = \pi(gh)^{-1} g \pi(h) y$ and used the left invariance of $m$.  But $F=0$ a.e. with respect to $\nu_1 \times \nu_0 \times m$, so we have $\nu(E)=0$.
\end{proof}

\begin{theorem}\label{GH-implies-A}
  If $E \subset G$ is Gauss--Haar null, then it is Aronszajn null.
\end{theorem}

\begin{proof}
  We prove the contrapositive.  Suppose $E$ is not Aronszajn null, so there exists a Carnot spanning sequence $X_1, X_2, \dots$ with $E \notin \mathcal{A}(X_1, X_2, \dots)$.  For each $k$, let $H_k = \langle X_1, \dots, X_k \rangle$, which by assumption is a Carnot group.  Then fix on $H_k$ a probability measure $\nu_k$ which is absolutely continuous with respect to Haar measure, has strictly positive density, and whose projection $\gamma_k = \pi_{\ast} \nu_k$ is a (centered) non-denegerate Gaussian measure on $W_k = \pi(H_k) \subset H_k \cap W$.  For instance, $\nu_k$ could be a non-degenerate Gaussian measure on $H_k$ viewed as a finite-dimensional vector space; or $\nu_k$ could be a heat kernel measure (see Section \ref{heat-kernel}).

  For $\epsilon > 0$, let $\nu_k^\epsilon = (\delta_\epsilon)_{\ast} \nu_k$, so that by Lemma \ref{l:ac-aoi}, the family $\{\nu_k^\epsilon : \epsilon > 0\}$ satisfies \eqref{eq:approximation}.  By Theorem \ref{th:construction}, there exists a sequence $\epsilon_1, \epsilon_2, \dots$ and an element $g \in E$ such that the sequence of measures $\nu^{\epsilon_1, \dots, \epsilon_n} = \nu_n^{\epsilon_n} \ast \dots \ast \nu_1^{\epsilon_1}$ converges weakly to a measure $\nu$, for which $(g \nu)(E) = \nu(g^{-1} E) > 0$.   To complete the proof, we show that $\gamma = \pi_{\ast} \nu$ is a non-degenerate Gaussian on $W$, and that $\nu \ll \gamma \ast m$.

Observe that since $\pi$ is a continuous homomorphism, $\gamma$ is the weak limit of the sequence $\pi_{\ast} \nu^{\epsilon_1, \dots, \epsilon_n} = \gamma_1 \star \dots \star \gamma_n$, and so $\gamma$ is Gaussian because every $\gamma_k$ is.  To see it is non-degenerate, let $f \in W^{\ast}$ be arbitrary and suppose $q(f,f) = 0$, so that $f=0$, $\gamma$-a.e.  We claim $f$ must be the zero functional.  For every $k$, by properties of convolution and Fourier transforms, $f$ must be constant $\gamma_k$-a.e.  But $\gamma_k$ is a non-degenerate Gaussian measure on the linear subspace $W_k \subset W$, so we must have $W_k \subset \ker f$.  Since $f$ now vanishes on the dense linear subspace $\bigcup_k W_k$, it is the zero functional.

  Next, since $X_1, X_2, \dots$ is Carnot spanning and $[G,G]$ is finite-dimensional, there exists some large enough $N$ so that $[H_N, H_N] = [G,G]$.  Let $\nu' = \nu^{\epsilon_1, \dots, \epsilon_N}$, and let $\nu''$ be the weak limit of $\nu_{M}^{\epsilon_M} \ast \cdots \ast \nu_{N+1}^{\epsilon_{N+1}}$ as $M \to \infty$, so that $\nu = \nu'' \ast \nu'$.  Since $\nu_{N}^{\epsilon_N}$ is an absolutely continuous measure on $H_N$ with strictly positive density, and $\gamma_N \times m$ is another one, they are mutually absolutely continuous.  Hence Lemma \ref{ac-inf} applies, and we conclude that $\nu \ll \pi_{\ast} \nu \times m = \gamma \times m$, which completes the proof.
\end{proof}

\begin{remark}
When $G$ is a Banach space (an abelian infinite-dimensional Carnot group), Theorem~\ref{GH-implies-A} recovers the main result of \cite{Csornyei1999a}.  Even after making the obvious simplifications, the proof is still somewhat more complicated than necessary.  In the abelian case, the \emph{retreading} is not needed, and instead of constructing the desired Gaussian measure as a CAC measure, we could instead construct it simply as a convolution of one-dimensional Gaussians on the axes $Y_i$.
\end{remark}

\section{Heat kernel measures} \label{heat-kernel}

We continue the assumptions and notation of the previous section.  Let $G$ still be an infinite-dimensional Carnot group with locally compact commutator subgroup, which is thus a stratified Banach--Lie group $G =   W \times \mathbf{C}$.

In this section, we discuss hypoelliptic heat kernel measures \cite{BaudoinGordinaMelcher2013,DriverEldredgeMelcher2016}, a family of measures on $G$ that in some ways are even more naturally compatible with its structure than the Gauss--Haar measures defined above in Definition \ref{def:gauss-haar}.   The definition of these measures requires some background from stochastic calculus, which we will not review here, but refer the reader to \cite{DriverGordina2008, Melcher2009a} and for infinite-dimensional non-nilpotent groups to \cite{Gordina2000a, Gordina2000b, Gordina2005b}.

Since the commutator subgroup of $G$ is assumed locally compact, then $G$ is a Banach Lie group (see Theorem~\ref{th:IDCG-Banach}) we can, as previously done in the previous section, identify $G$ with its Lie algebra $\mathfrak{g}$, considering them as the same set, without writing the exponential map.
Moreover, we consider the Banach space $W = V_1 = V_1(G)$, as the first layer of $\mathfrak{g} = G$. Let $\gamma$ be a non-degenerate Gaussian measure on $W$. Then there exists a corresponding Brownian motion process $B_t$ on $W$, which is a continuous Markov process with stationary independent increments such that $(t-s)^{-1/2} (B_t-B_s) \sim \gamma$ for every $t > s > 0$.  The hypoelliptic Brownian motion $g_t$ on $G$ is defined as the solution of the Stratonovich-type stochastic differential equation
\begin{equation}\label{sde}
  \circ dg_t = dL_{g_t} \circ dB_t, \qquad g_0 = e_G.
\end{equation}
It is actually possible to solve \eqref{sde} explicitly in terms of iterated stochastic integrals, by using the Baker--Campbell--Hausdorff--Dynkin formula as in \cite{Melcher2009a}.

Now let $\nu = \operatorname{Law}(g_1)$ be the endpoint distribution of $g_t$ at time $t=1$.  We say $\nu$ is the \emph{(hypoelliptic) heat kernel} on $G$ induced by $\gamma$.

The choice of time $1$ is of course arbitrary, and involves no loss of
generality because of the following time-space scaling property:
\begin{equation}\label{time-space}
  \nu \ast \nu = \operatorname{Law}(g_2) = (\delta_{\sqrt{2}})_{\ast} \nu.
\end{equation}
Equation \eqref{time-space} also provides justification for taking the heat kernel measures as our analogue of Gaussian measures, since the scaling \eqref{time-space} also holds for Gaussian measures on a Banach space, and indeed uniquely characterizes them.  See \cite[Theorem 1.9.5]{BogachevGaussianMeasures}, where the result is attributed to \cite{Polya1923}.

\begin{definition}
We will say a Borel set $E \subset G$ is \emph{heat kernel null} if we have $\nu(g^{-1}E) = 0$ for every heat kernel measure $\nu$ on $G$ and every $g \in G$.
\end{definition}

\begin{remark}
  In \cite{DriverGordina2008, Melcher2009a} the main objects of study were \emph{elliptic} heat kernel measures, which are constructed from a Gaussian measure $\gamma$ supported on all of $\mathfrak{g}$, instead of only on $W$.  Elliptic heat kernel measures are known to satisfy stronger regularity conditions, but they are not compatible with the scalable group structure in the same way; for instance, \eqref{time-space} does not hold for elliptic heat kernel measures.   As such, in this paper we focus only on \emph{hypoelliptic} Brownian motion and heat kernels as defined above. The issue is analogous to the distinction between left-invariant Riemannian and sub-Riemannian metrics on the 3-dimensional Heisenberg group: the Riemannian geometries are easier to study with classical tools, but they do not respect the dilation, and so in that sense they are less natural.
\end{remark}

Until further notice, suppose that $G$ has step $2$, so that it is an infinite-dimensional Heisenberg-like group (Example \ref{ex.DGgroups}).  In this case, we are able to prove the following theorem.
\begin{theorem}\label{A-iff-HK}
Let $G$ be an infinite-dimensional Heisenberg-like group.  A Borel set $E \subset G$ is Aronszajn null if and only if it is heat kernel null.
\end{theorem}

The proof follows from the developments of the previous section, together with some critical results on the regularity of the heat kernel.

\begin{theorem}[{\cite[Theorem 5.9]{BaudoinGordinaMelcher2013}, see also \cite[Section 7]{DriverEldredgeMelcher2016}}] \label{heis-qi} Let $\nu$ be the heat kernel measure induced by a Gaussian measure $\gamma$ on $W$, and let $H \subset W$ be the Cameron--Martin space for $\gamma$.  Define the \emph{Cameron--Martin subgroup} $G_{CM} = H \times C$, which is a dense scalable subgroup of $G$.  Then $\nu$ is quasi-invariant under left and right translation by $G_{CM}$, and in particular by $H$.
\end{theorem}

By Lemma \ref{A-null-qi} it immediately follows that
\begin{corollary}\label{A-implies-HK}
If $E \subset G$ is Aronszajn null, then it is heat kernel null.
\end{corollary}

For the converse direction, we make use of a result from \cite{DriverEldredgeMelcher2016} that was originally a lemma for a proof of Theorem \ref{heis-qi}, but is also potentially useful in its own right.

\begin{theorem}[{\cite[Corollaries 4.7 and 4.8]{DriverEldredgeMelcher2016}}]\label{HK-GH} Let $\nu$ be the heat kernel measure induced by a Gaussian measure $\gamma$ on $W$.  Then $\nu$ is absolutely continuous to the Gauss--Haar measure $\gamma \times m$.  Moreover, the density is positive almost everywhere, so that in fact $\nu$ and $\gamma \times m$ are mutually absolutely continuous.
\end{theorem}

\begin{corollary}\label{HK-implies-A}
  If $E \subset G$ is heat kernel null, then it is Aronszajn null.
\end{corollary}

\begin{proof}
  Suppose $E$ is not Aronszajn null.  By Theorem \ref{GH-implies-A}, there is a Gaussian measure $\gamma$ on $W$  and some $g \in G$ such that $(\gamma \times m)(g^{-1}E) > 0$.  Let $\nu$ be the heat kernel measure induced by $\gamma$.  By Theorem \ref{HK-GH}, we have $\gamma \times m \ll \nu$, so that $\nu(g^{-1}E) > 0$ as well.
\end{proof}

Combining Corollaries \ref{A-implies-HK} and \ref{HK-implies-A}, we have proved Theorem \ref{A-iff-HK}.

\section{Open questions} \label{s.Questions}

To conclude the paper, we collect some questions and topics for further research.

\subsection{Other notions of null sets}
  Having shown the equivalence of two notions of null sets, one might consider others.
\begin{question}
 What other \emph{interesting} left-invariant $\sigma$-ideals of Borel sets can one construct on an infinite-dimensional Carnot group, and how do they relate to the Aronszajn / filtration / CAC null $\sigma$-ideal $\mathcal{A}$ studied here?
\end{question}
By \emph{interesting} we mean that the $\sigma$-ideal ought to be defined using the scalable group structure in a meaningful way, unlike for instance the $\sigma$-ideal of countable sets.

\subsection{Stratified Banach-Lie groups}

It is still unclear when infinite dimensional Carnot groups are or are not stratified Banach-Lie groups.  The assumption in Theorem \ref{th:IDCG-Banach} that the commutator subgroup is locally compact seems fairly restrictive.
\begin{question}
  Under what conditions is an nilpotent infinite dimensional Carnot group $G$ isomorphic to a stratified Banach-Lie group?
\end{question}
For instance, perhaps it is sufficient that $[G,G]$ be homeomorphic to a Banach space.  The example of a non-Banach manifold infinite dimensional Carnot groups in \cite{LeDonneLiMoisala2021} has a commutator subgroup that is not homeomorphic to a Banach space.

\subsection{Aronszajn null sets and their compact subsets}\label{aronszajn-tight}

As an infinite-dimensional Carnot group $G$ is a Polish space, every Borel probability measure on $G$ is Radon \cite[Theorem 7.1.7]{BogachevMeasureTheory}, so it follows that a Borel set $E \subset G$ is CAC null if and only if every compact subset of $E$ is CAC null.  Thus, Theorem \ref{th:aron-cac} has the following interesting corollary:
\begin{corollary}\label{aronszajn-tight-cor}
  A Borel set $E \subset G$ is Aronszajn null if and only if every compact subset of $E$ is Aronszajn null.
\end{corollary}

If Corollary \ref{aronszajn-tight-cor} could be proved \emph{a priori}, then in proving Theorem \ref{th:construction}, we could assume without loss of generality that the set $E$ is compact.  In this case, the proof of Theorem \ref{th:construction} could be simplified somewhat, and the technical Lemma \ref{borel-decomp} would not be needed.

\begin{question}
  Is there a self-contained proof of Corollary \ref{aronszajn-tight-cor}?
\end{question}

This would be interesting even in the setting of Banach spaces.

\subsection{Heat kernel null sets in higher step}

The only reason that Theorem \ref{A-iff-HK} is restricted to infinite-dimensional Heisenberg--like groups is that the regularity results Theorems \ref{heis-qi} and \ref{HK-GH} have, to date, only been proven in that setting.  However, we  conjectured that they should hold more broadly; for instance, in every stratified Banach-Lie group with finite-dimensional commutator subgroup (the setting of Section \ref{s.HM}).  We do not know of any fundamental obstructions to such a result; it is mainly that working in step 3 or higher introduces technical difficulties and much more extensive computation.  As such, we pose the following conjecture.
\begin{conjecture}\label{conj.7.8}
    In every stratified Banach-Lie group $G$ with finite-dimensional commutator subgroup, the class of Aronszajn null sets and the class of heat kernel null sets coincide.
\end{conjecture}
Indeed, it is even possible that the assumption of finite-dimensional commutator subgroup is unnecessary, or even the assumption of nilpotency, but these cases seem further out of reach.

\bibliographystyle{plain}
\bibliography{EGLLD}

\end{document}